\RequirePackage{fix-cm}
\documentclass[smallextended]{svjour3}       
\smartqed  
\usepackage{graphicx}
\usepackage{float}
\usepackage{amsfonts}
\usepackage{amsmath}
\usepackage{amsbsy}
\usepackage{theorem}
\usepackage{epstopdf}
 \journalname{Test}
\begin{document}

\title{Dynamical  multiple   regression  in function spaces, under   kernel regressors, with ARH(1) errors
}

\titlerunning{Dynamical multivariate funtional multiple regression}        

\author{M. D. Ruiz-Medina         \and
        D. Miranda \and   R.M. Espejo
}

\authorrunning{Ruiz-Medina, Miranda and Espejo} 

\institute{M.D. Ruiz-Medina \at
              Department of Statistics and O.R. University of Granada \\Campus Fuente Nueva s/n, Granada,  18071, SPAIN\\
              Tel.: +(34)958243270\\
              Fax: +(34)958243267\\
              \email{mruiz@ugr.es}           
           \and
           D. Miranda \at
              Department of Statistics and O.R. University of Granada \\Campus Fuente Nueva s/n, Granada,  18071, SPAIN\\
\and
R.M. Espejo\at
Department of Statistics and O.R. University of Granada \\Campus Fuente Nueva s/n, Granada,  18071, SPAIN\\
}

\date{Received: date / Accepted: date}

\maketitle

\begin{abstract}
A  linear multiple regression model in function spaces is formulated, under
temporal correlated errors.  This formulation involves kernel regressors.    A  generalized least-squared regression parameter estimator is derived. Its asymptotic normality and strong consistency is obtained, under suitable conditions.    The  correlation  analysis  is based on a   componentwise  estimator of the residual  autocorrelation operator. When the dependence structure of the functional error term is unknown, a plug-in generalized least-squared regression parameter estimator is  formulated. Its strong-consistency is proved as well.      A simulation study is undertaken to illustrate  the performance  of the presented  approach, under  different regularity conditions. An application to  financial panel data is also considered. 
\keywords{ARH(1)   errors \and dynamical functional multiple  regression  \and firm leverage maps  \and generalized least squared estimator
\and kernel regressors}
\subclass{MSC code1 60G25\and  60G60\ and  62J05   \and MSC code2 62J10}
\end{abstract}
\section{Introduction}
\label{intro}
 Several authors highlight the advantages of the functional regression framework over discrete multivariate approaches (see, for example,  Marx and  Ei-\linebreak lers, 1999; Ramsay  and  Silverman, 2005;  Cuevas,   Febrero and   Fraiman,  2002). Indeed, only in the functional  setting, we can incorporate  smoothness assumptions on the predictors, and  the regression  parameter space. In particular, Crambes, Kneip and  Sarda  (2009) derive a smoothing splines estimator for the functional slope parameter. They  prove that the rate of convergence of the prediction
error depends on the smoothness of the slope function, and on the structure of
the predictors. An  overview on    functional principal component regression and functional partial
least-squared regression, in the parameter estimation of the functional linear model with scalar response, is presented in Febrero-Bande,  Galeano and  Gonzalez-Manteiga (2015). There exists an extensive  literature on the asymptotic properties of functional regression estimators, in the case of  scalar response and functional regressors (see, for example, Cai and  Hall, 2006, and the references therein). 
A semi-functional partial linear approach for regression, based on nonparametric time series, is  considered in Aneiros-P\'erez and Vieu (2006; 2008). 
 Applying the Projection Pursuit Regression
principle,     the approximation of the regression
function in the case of a functional predictor and
a scalar response is addressed in Ferraty et al. (2013)   (see also 
 Ferraty and Vieu, 2006; Ferraty and Vieu,  2011).  In the nonparametric setting, the case of functional response and predictor is studied, for example,  in   Ferraty,  Keilegom and Vieu  (2012), where  a kernel type estimator of the regression operator is derived, and its pointwise asymptotic normality is obtained. Goia and  Vieu (2015) adopt a semiparametric approach,  in  a two-terms Partitioned Functional Single Index Model.  Cuevas  (2014)   discusses central topics in Functional Data Analysis (FDA), related to probabilistic tools,  definition and estimation of centrality parameters, and the main trends in regression, classification,  
dimension reduction, and bootstrap methods for FDA.
   Recent advances  in the statistical analysis of 
 high-dimensional data, including regression, from the parametric, semiparametric  and nonparametric FDA frameworks,  are collected in the Special Issue  by Goia and Vieu (2016).

 The   kernel  formulation of the  regression parameters is usually  adopted in the literature of parametric linear  regression with  functional response and regressors (see, for example,   Chiou,  M\'uller and  Wang, 2004; Ruiz-Medina, 2011; Ruiz-Medina, 2012a; Ruiz-Medina, 2012b,  and the references therein).
  An extensive review, and further references for functional regression approaches,
including the case of functional response and regressors, can be found in Morris (2015). See also the   monograph by  Hsing and Eubank (2015), where several  functional analytical tools are introduced, for the estimation of  random elements in function spaces.  
 The concept of $L^{r}-m$-approximable processes also allows to modeling  the temporal dependence in the regression functional errors (see, for example,  Horv\'ath and  Kokoszka, 2012).  A central topic in this book is the analysis of functional data,  displaying  dependent  structures    in time and space.
A fixed effect approach in Hilbert spaces is adopted in Ruiz-Medina (2016), for FANOVA analysis under dependent errors. For simple regression, with explanatory variable taking  values in some abstract space of functions, the rate of convergence of the mean squared error of the functional version of the Nadaraya--Watson kernel estimator
is derived, in  Benhenni,  Hedli-Griche and Rachdi  (2017),  when the errors are represented by a  stationary short or long memory process. 

The present paper considers  functional response and kernel regressors,  and adopts the ARH(1) process framework (see  Bosq, 2000), to represent the temporal correlation  of the functional errors. The  efficiency, consistency and asymptotic normality  of a  componentwise estimator of the residual autocorrelation operator can then be  obtained,  from the results derived, in the ARH(1) process framework (see, for example,  Bosq, 2000;  Bosq and Ruiz-Medina, 2014; Guillas,  2001;  Mas,   2004; and  Mas,  2007). The  nonparametric time series model  introduced in Ferraty,  Goia and  Vieu (2002) could also be adopted in the representation of the temporal dependence displayed by  the regression  error term. However, this paper focuses in the linear parametric time series framework. As proved in this paper, good asymptotic properties are displayed by the regression estimators in this framework,  avoiding, in particular,  some computational drawbacks,  arising  in the  nonparametric functional statistical context. It is well-known that  the functional nonparametric  statistical modelling offers a   more flexible framework, but suffers of the so-called \emph{curse of dimensionality}, caused by the sparsity of data in high-dimensional spaces, affecting the asymptotic properties, in particular,  of the nonparametric regression estimators. Geenens (2011) proposes slightly modified estimators, considering a semi-metric to measure the  proximity between two random elements in an infinite-dimensional space. Furthermore, the implementation of nonparametric estimators  requires   the resolution of several selection problems. For example, in the implementation of the local-weighting-based  approach,   a smoothing parameter, and  a suitable  kernel must be previously selected. Recently, Kara et al. (2017a)
 investigate various nonparametric models, including regression, conditional distribution, conditional density and conditional hazard function, when the covariates are infinite dimensional. They  prove uniform in bandwidth asymptotic results for kernel estimators of these functional operators. Data-driven bandwidth selection is also discussed for applications.

Inverse problems can be described as functional equations, where the value of the
function is known or easily estimable, but the argument is unknown. In the finite-dimensional case, parameter estimation of the general linear model constitutes an example of inverse problem, where the unknown  argument of the design matrix, the regression parameter,  should be approximated. The usual two-dimensional definition  of the design matrix involves the sample, and    the covariate population dimensions. 
In the analysis of functional data, more complex dependence models arise, involving   conditional distributions in abstract spaces. 
We refer to the reader to the  recent contribution by Chaouch, Laib and Louani (2017), on  kernel conditional mode estimation, from functional stationary ergodic data, 
in the context of random elements in  semi-metric abstract spaces (see also Ling, Liu and  Vieu, 2017).

 This paper considers the problem of linear functional multiple regression estimation, when  the   response takes values in an abstract separable Hilber space $H,$ and the  regressors are operators on $H.$ The temporal dependence of the  errors is represented, in terms of an ARH(1)  time series model. Indeed, the presented  approach provides a  functional formulation of the parametric part, appearing in the  
 above-referred semiparametric model adopted in 
Aneiros-P\'erez and  Vieu (2006; 2008) (but  under a  parametric framework, in the linear  time series analysis of the temporal correlated random part). 
 
  The practical motivation of the    kernel formulation of the regressors relies on the incorporation of possible correlations between the response and the regressors at different scales and domains in  time, space or  depth, among others.
   For example, the designed experiments could  be run over time, with the  control of the regressors   over  space and depth, in a period of time.  
 This type of models arise, for instance, in the estimation of ocean surface temperature maps over time, from the evolution of related  functional covariates observed at different ocean depth intervals (see Espejo, Fern\'andez-Pascual and Ruiz-Medina, 2017).    In  this paper, a financial panel data set is analyzed.   Firm leverage  mapping, during a given period of time, in the Spanish communities of the Iberian Peninsula, is addressed from a  functional perspective. The kernel regressors are the firm factor determinants, involved in the analysis of the financing decisions of the company, depending on  the industrial area sampled,  and the  Spanish community studied 
 (see  Section \ref{application} and Supplementary Material II).
 The proposed functional estimation approach involves two steps: 
Generalized least-squared regression parameter estimation, and  ARH(1) residual correlation analysis, for functional estimation of the response. The strong consistency of the  generalized 
least-squared functional regression parameter estimator is derived. In the case where the auto-covariance matrix operator of the error term is unknown, the strong-consistency of  the corresponding generalized 
least-squared plug-in  estimator is obtained as well.  Asymptotic normality of the generalized-least-squared functional parameter estimator is  proved, in the case where the functional errors follow a known infinite-dimensional  Gaussian distribution. 

 The outline of the paper is the following.
Section \ref{sec3} introduces the studied dynamical multiple regression model in Hilbert spaces, with  ARH(1) error term. 
The generalized least-squared   regression parameter estimator     is derived 
in Section \ref{sec3b}. Its asymptotic normality and strong consistency is obtained as well.  When the functional correlation structure of the error process is unknown, sufficient conditions are  considered, for the strong consistency of the  generalized 
least-squared plug-in parameter  estimator, in Section \ref{esresiduals}.
A simulation study is undertaken in Section   \ref{Secsims}, to illustrate the performance of the presented approach, under different  scenarios, assuming different regularity conditions on the regression  functional  parameters, kernel regressors, and error correlation structure.
A real-data application is developed in Section \ref{application}, in the financial panel data context.  Final comments are provided  in Section \ref{fc}. The finite-sample-size properties  in relation to the truncation parameter are also illustrated, in  the  simulation study (see also  Supplementary Material I). In the  Supplementary Material II, details on in the real-data application, and   the practical implementation  are provided.
\section{The model}
\label{sec3}
Let $(\Omega,\mathcal{A},P)$  be  the basic probability space, and $H$ be a real separable Hilbert space. 
The following inverse problem formulation  of a  dynamical functional regression model is studied:
\begin{equation}
Y_{n}=X_{n}^{1}(\beta_{1}) +\dots + X_{n}^{p}(\beta_{p})+\varepsilon_{n},\quad n\in \mathbb{Z},
\label{modelreg}
\end{equation}
\noindent where $\boldsymbol{\beta}=(\beta_{1}(\cdot),\dots \beta_{p}(\cdot))^{T}\in H^{p};$ $X_{n}^{j}\in \mathcal{S}(H),$  $j=1,\dots,p,$ $n\in \mathbb{Z},$ with $\mathcal{S}(H)$ being  the Hilbert space  of Hilbert--Schmidt operators on $H,$ and $Y_{n},$ $\varepsilon_{n}\in H,$ for each $n\in \mathbb{Z}.$
For a given orthonormal basis $\{\varphi_{k}\}_{k\geq 1}$ of $H,$ denote 
\begin{equation}\left\langle X_{n}^{j}(\varphi_{k}),\varphi_{l}\right\rangle_{H}=x^{j}_{k,l}(n),\quad k,l\geq 1,\ \forall n\in \mathbb{Z},\quad j=1,\dots,p.\label{regressors}\end{equation}
\noindent Since $X_{n}^{j}\in \mathcal{S}(H),$    then, $\sum_{k,l}[x^{j}_{k,l}(n)]^{2}<\infty,$ 
\begin{equation}X_{n}^{j}(f)\underset{H}{=}\sum_{k,l}x^{j}_{k,l}(n)\left\langle f,\varphi_{l}\right\rangle_{H}\varphi_{k},\quad \forall f\in H,\label{regressors2}\end{equation}
\noindent  for every  $n\in \mathbb{Z},$ $j=1,\dots,p,$ where $\underset{H}{=}$ means the equality in the norm of $H.$

The  error term  $\varepsilon\equiv \{\varepsilon_{n},\ n\in \mathbb{Z}\}$ satisfies 
\begin{equation}
E\left[\varepsilon_n |  X_{n}^{1},\dots, X_{n}^{p}\right]=0,\quad \forall n\in \mathbb{Z}.
\label{erruncorrelated} 
\end{equation}
Furthermore, $\varepsilon $ is assumed to be a  zero-mean  ARH(1) process,  i.e., 
\begin{equation}\varepsilon_n=\rho(\varepsilon_{n-1})+\delta_{n}, \ n\in\mathbb{Z},\label{ARHerroerterm}
\end{equation}
\noindent where $\rho$ denotes the autocorrelation operator, which belongs to the space of bounded linear operators $\mathcal{L}(H)$ on $H,$ satisfying 
$\|\rho\|_{\mathcal{L}(H)}^{k}<1,$ for $k\geq k_{0},$ for certain $k_{0}\in \mathbb{N}.$ Here, 
$\{\delta_{n},\ n\in \mathbb{Z}\}$ is a sequence of independent and identically distributed $H$-valued  
zero-mean random variables, with trace autocovariance operator, i.e., defining strong-white noise in $H.$ They are  uncorrelated with the random  initial condition
$\varepsilon_{0}$ (see Bosq, 2000).
\begin{remark} Let $\{\varepsilon_{n},\ n\in \mathbb{Z}\}$ be Gaussian  with known auto-covariance and cross-covariance  operators,  the generalized least-squared estimator $\widehat{\boldsymbol{\beta}}_{N},$  derived in Section \ref{sec3b} below,   displays an asymptotic 
infinite-dimensional Normal  distribution, as the functional sample size $N\to \infty.$ 

 A generalization of the classical linear statistical test, for checking    the significance of the functional  parameters $\boldsymbol{\beta}_{1},\dots,\boldsymbol{\beta}_{p}$ can be obtained (see, for example, Theorem 3 in  Section 6, in Ruiz-Medina, 2016). Indeed, under this Gaussian scenario,  the adaptative selection, in time, of the regressors    could be derived, from  a temporal  adaptative significance statistical test, keeping in mind the ARH(1) structure of the  error term (see, for example,  Kara et al.,  2017b, where the same ideas motivate the use of kernel Nearest-Neighbor (kNN) estimators,  in the nonparametric  regression, conditional density, conditional
distribution, and hazard operator based estimation). 
\end{remark}
 Denote by  $$R_{0}=E[\varepsilon_{0}\otimes \varepsilon_{0}]=E[\varepsilon_{n}\otimes\varepsilon_{n}],\quad \forall n\in \mathbb{Z},$$ \noindent the trace autocovariance operator, and   by
$$R_{1}=E[\varepsilon_{0}\otimes\varepsilon_{1}]=E[\varepsilon_{n}\otimes \varepsilon_{n+1}],\quad \forall n\in \mathbb{Z},$$ \noindent the nuclear 
cross-covariance operator.

The experiment is run, and  a functional sample $Y_{1},\dots, Y_{N}$ of size $N$ of the response (\ref{modelreg}) is collected, under the control of  the kernel regressors, $X_{i}^{1},\dots, X_{i}^{p},$   for the times $i=1,\dots, N.$    From equations  (\ref{modelreg}), and  (\ref{erruncorrelated})--(\ref{ARHerroerterm}),
 \begin{eqnarray}
&& \mu_{n,\mathcal{X}}=E[Y_{n}|X_{n}^{1},\dots, X_{n}^{p}]=X_{n}^{1}(\beta_{1}) +\dots + X_{n}^{p}(\beta_{p}),\quad 
 n=1,\dots,N\nonumber\\
&& E\left[(Y_{i}-\mu_{i,\mathcal{X}})\otimes (Y_{j}-\mu_{j,\mathcal{X}})\right]=
 E\left[ \varepsilon_{i}\otimes \varepsilon_{j}\right]
 = \rho^{|j-i|}R_{0},
 \label{eqresponse}
 \end{eqnarray}
\noindent   for  $i,j\in \{1,\dots,N\},$ where, $\mathcal{X}$ denotes the  vector  value of the covariates, to which we are conditioning. Here,   in the last equation,  we have applied that 
$$\varepsilon_{n}=\sum_{j=0}^{k}\rho^{j}\delta_{n-j}+\rho^{k+1}(\varepsilon_{n-k-1}),\quad k\geq 1$$
\noindent (see equation (3.11) in Bosq, 2000).
 Thus, the 
 covariance structure of the functional errors  $Y_{1}-\mu_{1,\mathcal{X}},\dots, Y_{N}-\mu_{N,\mathcal{X}}$ can be expressed, in matrix operator form,  as follows:
 \begin{eqnarray}\mathbf{C}&:=&E\left[\left((Y_{1}-\mu_{1,\mathcal{X}}),\dots,   (Y_{N}-\mu_{N,\mathcal{X}}) \right)^{T}\right.\nonumber\\ && \hspace*{1cm}\otimes \left.\left((Y_{1}-\mu_{1,\mathcal{X}}),\dots,  (Y_{N}-\mu_{N,\mathcal{X}})\right)\right]\nonumber\\
  &=&\left[\begin{array}{ccccc}R_{0} & \rho R_{0} &       \rho^{2} R_{0} & \ldots & \rho^{N-1} R_{0} \\
 \rho R_{0} & R_{0} & \rho R_{0} & \ldots & \rho^{N-2} R_{0}\\
\vdots & \ldots & \ldots & \ldots & \vdots\\
\rho^{N-1} R_{0} & \rho^{N-2} R_{0}&\ldots   &\ldots & R_{0}\\
 \end{array}\right]=  \left[\begin{array}{ccccc} I & \rho  &       \rho^{2}  & \ldots & \rho^{N-1}  \\
 \rho  & I & \rho  & \ldots & \rho^{N-2} \\
\vdots & \ldots & \ldots & \ldots & \vdots\\
\rho^{N-1}  & \rho^{N-2} &\ldots   &\ldots & I\\
 \end{array}\right] \nonumber\\
 &\times &\left[\begin{array}{ccccc}R_{0} & 0 &       0& \ldots & 0\\
 0 & R_{0} & 0 & \ldots & 0\\
\vdots & \ldots & \ldots & \ldots & \vdots\\
0 & 0&\ldots   &\ldots & R_{0}\\
 \end{array}\right]= \boldsymbol{\rho}\mathbf{R}_{0},
 \label{cocopmat}
\end{eqnarray}
\noindent where $I$ denotes the identity operator on $H.$

\begin{remark}
\label{ARHpcase} 
The present approach can be easily extended to the case of an ARH(p), $p\geq 2,$ error term, replacing operator $\rho$ by $$\rho^{\prime }=\left[\begin{array}{cccc}
 \rho_{1} & \rho_{2} &\dots & \rho_{p}\\
 I & 0 & \dots & 0  \\
  0 & I & 0& \dots   \\
  0 &\dots & I & 0\\
  \end{array}\right],
  $$ 
  \noindent where, as before,  $I$ denotes the identity operator on $H$ (see Bosq, 2000, p.128).
\end{remark}

If   $\mathbf{C}^{-1}$ exists,  then \begin{equation}
\mathbf{C}^{-1}=\mathbf{R}_{0}^{-1}\boldsymbol{\rho}^{-1}.\label{invmatcovop}
\end{equation}
It is clear that $\mathbf{R}_{0}^{-1}$ exists if and only if $R_{0}^{-1}$ exists, where \linebreak $\mathbf{R}_{0}^{-1}:=\mbox{diag}\left(R_{0}^{-1},\dots,R_{0}^{-1}\right)_{N\times N},$ with $\mbox{diag}\left(R_{0}^{-1},\dots,R_{0}^{-1}\right)_{N\times N}$ denoting an $N\times N$  diagonal matrix operator with functional diagonal entries equal to $R_{0}^{-1}.$

 Denote  by $\{\phi_{k}\}_{k\geq 1}$ and $\{\lambda_{k}(R_{0})\}_{k\geq 1}$  the eigenvectors and eigenvalues  of $R_{0},$ respectively. The following assumptions are made:

\noindent \textbf{Assumption A1}. The systems of  eigenvalues  of $R_{0}$  satisfy  \linebreak $\lambda_{1}(R_{0})>\lambda_{2}(R_{0})>\dots > \lambda_{m}(R_{0})\dots>0.$

\medskip

 \noindent \textbf{Assumption A2}. The autocorrelation operator $\rho$ of the error term $\varepsilon $ is a self-adjoint compact operator on $H.$
  
\medskip
   
Under    \textbf{Assumption A1}, we can formally define the kernel $k_{R_{0}}$ of the inverse $R_{0}^{-1}$ of $R_{0}$ as 
   $k_{R_{0}}=\sum_{m=1}^{\infty}\frac{1}{\lambda_{m}(R_{0})}\phi_{m}\otimes \phi_{m}$  (see Dautray and Lions, 1985, pp. 112-126).
   Since $R_{0}$ is a trace operator, $\frac{1}{\lambda_{k}(R_{0})}\to \infty,$ as $k\to \infty.$ Hence, a suitable orthonormal basis of $H$ in $R_{0}^{1/2}(H)$ must be found in order to explicitly compute $R_{0}^{-1}(f),$
   for every $f\in H.$ Otherwise, $R_{0}^{-1}$ can only be defined on the Reproducing Kernel Hilbert space (RKHS) $R_{0}^{1/2}(H)$ of $\varepsilon_{n},$  $n\in \mathbb{Z}$ (see Bosq, 2000; Da Prato and   Zabczyk, 2002, Chapter 1, pp. 12--16). 

Under \textbf{Assumption A2}, consider the system of eigenvectors  
$\{\psi_{k}\}_{k\geq 1}$ of the autocorrelation operator $\rho $ satisfying 
\begin{eqnarray}
&&\rho(\psi_{k})= \lambda_{k}(\rho)\psi_{k},\ k\geq 1;\quad
\rho(g)=\sum_{k=1}^{\infty}\lambda_{k}(\rho)\left\langle g,\psi_{k}\right\rangle_{H}\psi_{k},\ \forall g\in H.
\label{eq2asmodel}
\end{eqnarray}
\begin{lemma} \label{lemdr}
Let $\boldsymbol{\rho}$ be the matrix operator introduced in   (\ref{cocopmat}).
Under \textbf{Assumption A2}, $\boldsymbol{\rho}$ admits the following series representation in $H^{N}:$  For every  $\mathbf{f}=(f_{1},\dots,f_{N})^{T},$
\begin{eqnarray}
\boldsymbol{\rho}(\mathbf{f})&=&
\sum_{k\geq 1}\boldsymbol{\Psi}_{k}\left[\begin{array}{cccc}1 & \lambda_{k}(\rho) &        \ldots & \left[\lambda_{k}(\rho)\right]^{N-1}  \\
 \lambda_{k}(\rho)  &1 &  \ldots & \left[\lambda_{k}(\rho)\right]^{N-2}\\
\vdots & \ldots &  \ldots & \vdots\\
\left[\lambda_{k}(\rho)\right]^{N-1}
  & \ldots   &\ldots &1\\
 \end{array}\right]\boldsymbol{\Psi}_{k}^{\star}(\mathbf{f}),
\end{eqnarray}
\noindent  where 
for   $\mathbf{g}=
(g_{1},\dots,g_{N})^{T}\in H^{N},$ and $k\geq 1,$
\begin{eqnarray}\boldsymbol{\Psi}_{k}^{\star}(\mathbf{g})&:=&
\mathrm{diag}\left( \psi_{k},\dots, \psi_{k}\right)_{N\times N}(\mathbf{g})=\mathbf{g}_{k}
\nonumber\\
\boldsymbol{\Psi}_{k}\boldsymbol{\Psi}_{k}^{\star}(\mathbf{g})&=&
\boldsymbol{\Psi}_{k}(\mathbf{g}_{k})=\left[\begin{array}{c} \left\langle g_{1},\psi_{k}\right\rangle_{H}\psi_{k}\\
\left\langle g_{2},\psi_{k}\right\rangle_{H}\psi_{k}\\ \vdots\\ 
  \left\langle g_{N},\psi_{k}\right\rangle_{H}\psi_{k}\end{array}\right]\nonumber\\
   \boldsymbol{\Psi}_{k}^{\star}\boldsymbol{\Psi}_{k}
   &=&\mathrm{diag}\left(\left\langle \psi_{k},\psi_{k}\right\rangle_{H},\dots, \left\langle \psi_{k},\psi_{k}\right\rangle_{H}\right)_{N\times N}=I_{N\times N}.
\label{cocopmat3}
\end{eqnarray}

\noindent Here, $\mathbf{g}_{k}=\left( \left\langle g_{1},\psi_{k}\right\rangle_{H},\dots, \left\langle g_{N},\psi_{k}\right\rangle_{H}\right)^{T},$ $k\geq 1,$ and, as before  $\mathrm{diag}\left( \cdots \right)_{N\times N}$ denotes an $N\times N$ functional diagonal matrix.  Also,   $[\cdot ]^{\star}$ stands for the adjoint of the matrix operator $[\cdot],$ and $I_{N\times N}$ denotes the $N\times N$ identity matrix.
\end{lemma}
\begin{proof} Under \textbf{Assumption A2}, from equation (\ref{eq2asmodel}), considering the identity  $$\rho^{j}=\sum_{k=1}^{\infty}\left[\lambda_{k}(\rho)\right]^{j}\psi_{k}\otimes \psi_{k},\quad j=1,\dots, N-1,$$
\noindent   in equation (\ref{cocopmat}), for 
$\mathbf{f}=(f_{1},\dots,f_{N})^{T},$ $\boldsymbol{\rho} $ can then be expressed as 
\begin{eqnarray}
\boldsymbol{\rho}(\mathbf{f})
 &=&\sum_{k\geq 1}\left[\begin{array}{cccc}
\psi_{k} & 0 & \ldots & 0\\
0 &\psi_{k} & \ldots &0\\
\vdots & \ddots & \ddots &\vdots \\
0 & \ldots & \ldots &\psi_{k}\\ 
\end{array}\right]\left[\begin{array}{cccc}1 & \lambda_{k}(\rho) &        \ldots & [\lambda_{k}(\rho)]^{N-1}  \\
 \lambda_{k}(\rho)  &1 &  \ldots & [\lambda_{k}(\rho)]^{N-2}\\
\vdots & \ldots &  \ldots & \vdots\\ \left[\lambda_{k}(\rho)\right]^{N-1}   & \ldots   &\ldots &1\\
 \end{array}\right]
\nonumber\\ 
&\times &\left[\begin{array}{cccc}
\psi_{k} & 0 & \ldots & 0\\
0 &\psi_{k} & \ldots &0\\
\vdots & \ddots & \ddots &\vdots \\
0 & \ldots & \ldots &\psi_{k}\\ 
\end{array}\right]^{\star}\left[\begin{array}{c}f_{1}\\f_{2}\\ \vdots\\ f_{N}\end{array}\right]\nonumber\\&& \nonumber\\
&=&\sum_{k\geq 1}\boldsymbol{\Psi}_{k}
\boldsymbol{\Lambda}_{k}\boldsymbol{\Psi}_{k}^{\star}(\mathbf{f}),  
 \label{cocopmat2}
\end{eqnarray}
\noindent where $$\boldsymbol{\Lambda}_{k}:=
\left[\begin{array}{cccc}1 & \lambda_{k}(\rho) &        \ldots & \left[\lambda_{k}(\rho)\right]^{N-1}  \\
 \lambda_{k}(\rho)  &1 &  \ldots & \left[\lambda_{k}(\rho)\right]^{N-2}\\
\vdots & \ldots &  \ldots & \vdots\\
\left[\lambda_{k}(\rho)\right]^{N-1}
  & \ldots   &\ldots &1\\
 \end{array}\right],\quad k\geq 1.$$

Then, 
\begin{eqnarray}
 \boldsymbol{\Psi}_{k}^{\star}\boldsymbol{\rho}\boldsymbol{\Psi}_{k}&=&
  \boldsymbol{\Lambda}_{k}=\mathbf{A}_{k}^{T}\mathbf{A}_{k},\quad k\geq 1.\nonumber
\end{eqnarray}
\noindent with \begin{equation}\mathbf{A}_{k}=\left[
 \begin{array}{ccccc}
  1 & \lambda_{k}(\rho) \ & \ \lambda_{k}^{2}(\rho) \ & \ \ldots \ &\ \lambda_{k}^{N-1}(\rho)\\
 0 \  &  \ \sqrt{1-\lambda_{k}^{2}(\rho)} \ &  \
 \frac{-\lambda_{k}^{3}(\rho)+\lambda_{k}(\rho)}{\sqrt{1-\lambda_{k}^{2}(\rho)}} \ & \ \ldots \ & \ \frac{-\lambda_{k}^{N}(\rho)+\lambda_{k}^{N-2}(\rho)}{\sqrt{1-\lambda_{k}^{2}(\rho)}}\\
 \vdots \ & \ \ldots \ & \ \ddots \ & \ \ldots \ & \ \vdots \\ 
 \vdots \ & \ \ldots \  & \ \ldots \ & \ \ddots \ & \ \vdots \\
0 \ & \ 0 \ & \ \ldots \ & \ \sqrt{1-\lambda_{k}^{2}(\rho)} \ & \ \frac{-\lambda_{k}^{3}(\rho)+\lambda_{k}(\rho)}{\sqrt{1-\lambda_{k}^{2}(\rho)}}\\
0 \ & \ 0  \ & \ \ldots  \ & \ \ldots  \ & \ \sqrt{1-\lambda_{k}^{2}(\rho)}
  \end{array}
  \right]\label{chfact}
  \end{equation}
\end{proof}

\begin{remark} 
 From  Lemma \ref{lemdr}, $\boldsymbol{\rho}$ admits an 
 infinite--dimensional block diagonal representation, with respect to the orthonormal matrix functional system 
$\{\boldsymbol{\Psi}_{k}\}_{k\geq 1},$ with matrix diagonal entries $\boldsymbol{\Lambda}_{k},$ $k\geq 1.$
 Equivalently,  for $k\geq 1,$
 \begin{eqnarray}
\boldsymbol{\Lambda}_{k}  &=& E\left[\left(\left\langle (Y_{1}-\mu_{1,\mathcal{X}}),\psi_{k}\right\rangle_{H},\dots,  \left\langle (Y_{N}-\mu_{N,\mathcal{X}}),\psi_{k}\right\rangle_{H}\right)^{T}\right.\\ && \hspace*{0.5cm}\times \left.\left(\left\langle (Y_{1}-\mu_{1,\mathcal{X}}),\psi_{k}\right\rangle_{H},\dots,  \left\langle(Y_{N}-\mu_{N,\mathcal{X}}),\psi_{k}\right\rangle_{H}\right)\right]\left[\boldsymbol{\Psi}_{k}^{\star}\mathbf{R_{0}}\boldsymbol{\Psi}_{k}\right]^{-1}\nonumber\\
&=&E\left[\left(\left\langle\varepsilon_{1},\psi_{k}\right\rangle_{H}\dots,\left\langle\varepsilon_{N},\psi_{k}\right\rangle_{H}\right)^{T}\left(\left\langle\varepsilon_{1},\psi_{k}\right\rangle_{H},\dots,\left\langle\varepsilon_{N},\psi_{k}\right\rangle_{H}\right)\right]\left[\boldsymbol{\Psi}_{k}^{\star}\mathbf{R_{0}}\boldsymbol{\Psi}_{k}\right]^{-1}.\nonumber \end{eqnarray}

\end{remark}
 The following lemma will be applied in the formal  definition of 
the  norm of the  RKHS of $\varepsilon,$  in model (\ref{modelreg}), defining  the quadratic loss function in equation  (\ref{lossfunction}) below, involved in the computation of the generalized least-squared estimator $\widehat{\boldsymbol{\beta}}_{N}$ of parameter $\boldsymbol{\beta },$ in the next section. 
 \begin{lemma}
 For $i,j=1\dots N,$ the functional entries $\widetilde{\rho}_{i,j}$ of   $\boldsymbol{\rho}^{-1}=\left(\widetilde{\rho}_{i,j}\right)_{i,j=1\dots N}$ are formally given by: 
\begin{eqnarray}
\widetilde{\rho}_{1,1} &=& \widetilde{\rho}_{N,N}=(I-\rho^{2})^{-1}\nonumber\\
\widetilde{\rho}_{i,i+1} &=&\widetilde{\rho}_{j,j-1}
=-(I-\rho^{2})^{-1}\rho,\quad i=1,\dots, N-1,\ j=2,\dots, N\nonumber\\
\widetilde{\rho}_{i,i}&=&
(I-\rho^{2})^{-1}(I+\rho^{2}),\quad i=2,\dots, N-1.
 \label{invmco}
 \end{eqnarray}
 \end{lemma}

  \begin{proof}  
  Operator $\boldsymbol{\rho}$ is invertible if and only if  $[\boldsymbol{\Lambda}_{k}]_{N\times N},$  is invertible, for $k\geq  1.$ The inverse $\boldsymbol{\rho}^{-1}$ then admits an infinite-dimensional block diagonal representation with respect to $\{\boldsymbol{\Psi}_{k}\}_{k\geq 1},$ with matrix diagonal entries   
$$\boldsymbol{\Lambda}_{k}^{-1}=\left[\begin{array}{cccc}1 & \lambda_{k}(\rho) &        \ldots & \left[\lambda_{k}(\rho)\right]^{N-1}  \\
 \lambda_{k}(\rho)  &1 &  \ldots & \left[\lambda_{k}(\rho)\right]^{N-2}\\
\vdots & \ldots &  \ddots & \vdots\\
\left[\lambda_{k}(\rho)\right]^{N-1}
  & \ldots   &\ldots &1\\
 \end{array}\right]^{-1}_{N\times N}=\left[\mathbf{A}_{k}^{T}\mathbf{A}_{k}\right]^{-1}= \mathbf{A}_{k}^{-1}[\mathbf{A}_{k}^{T}]^{-1},$$  \noindent   where  
\begin{equation}\mathbf{A}_{k}^{-1}=\frac{1}{\sqrt{1-\lambda_{k}^{2}(\rho )}}\left[\begin{array}{ccccc} \sqrt{1-\lambda_{k}^{2}(\rho)} & \  -\lambda_{k}(\rho) & 0   & \ldots    & 0  \\
 0  &1 & \  -\lambda_{k}(\rho) & \ldots  &  0 \\
 \vdots & \ldots &  \ldots &  \ddots & \vdots\\
0 &0 &  \ldots & 1 & \ -\lambda_{k}(\rho)  \\
0 &0 &  \ldots & \ldots &1\\
 \end{array}\right]_{N\times N},\ k\geq 1\label{eidbm}
\end{equation}
\noindent (see, for example,  Fitzmaurice et al., 2004). Thus,   $\boldsymbol{\rho}^{-1}$ in 
(\ref{invmatcovop})  admits the following series representation: For every $\mathbf{f}=(f_{1},\dots,f_{N})^{T}\in H^{N},$
\begin{eqnarray}
\boldsymbol{\rho}^{-1}(\mathbf{f})
   &=&\sum_{k\geq 1}\boldsymbol{\Psi}_{k}\boldsymbol{\Lambda}_{k}^{-1}\boldsymbol{\Psi}_{k}^{\star}(\mathbf{f}),\label{sriamo}
 \end{eqnarray}
 \noindent where, for each $k\geq 1,$ the $N\times N$ matrix 
 $\boldsymbol{\Lambda}_{k}^{-1}$ is given by
\begin{equation} 
\boldsymbol{\Lambda}_{k}^{-1}=\frac{1}{1-\lambda_{k}^{2}(\rho)} \left[\begin{array}{cccccc} 1 & \ -\lambda_{k}(\rho)  & 0 & \ldots & \ldots & 0  \\
-\lambda_{k}(\rho)  & \ 1+\lambda^{2}_{k}(\rho)& \ -\lambda_{k}(\rho)& 0 & \ldots &0  \\
 \vdots & \ldots &  \ldots & \ldots & \vdots & \vdots\\
 0& \ldots &  \ldots & -\lambda_{k}(\rho) & \ 1+\lambda^{2}_{k}(\rho)&\ -\lambda_{k}(\rho) \\
0 &0 &  \ldots & \ldots & -\lambda_{k}(\rho)  &1\\
 \end{array}\right].\label{rhostrproy} 
 \end{equation}

 From (\ref{sriamo})--(\ref{rhostrproy}), using Spectral Theorems on Spectral Calculus for continuous functions of self-adjoint operators on a Hilbert space (see Dautray and Lions 1985, pp. 112-126, for continuous functions,  and p. 140, for the unbounded case),
we obtain that, for $i,j=1\dots N,$ the functional entries $\widetilde{\rho}_{i,j}$ of   $\boldsymbol{\rho}^{-1}=\left(\widetilde{\rho}_{i,j}\right)_{i,j=1\dots N}$ are defined as in equation (\ref{invmco}).
 
\end{proof}  
  
\bigskip  
  
 From (\ref{invmatcovop})--(\ref{invmco}), the functional entries  $\widetilde{C}_{ij},$ $i,j=1,\dots,N,$ of \linebreak $\mathbf{C}^{-1}=\left(\widetilde{C}_{ij}\right)_{i,j=1\dots,N}$ are formally defined as
 \begin{eqnarray}
\widetilde{C}_{1,1} &=& \widetilde{C}_{N,N}=R_{0}^{-1}(I-\rho^{2})^{-1}\nonumber\\
\widetilde{C}_{i,i+1}&=&\widetilde{C}_{j,j-1}=  
-R_{0}^{-1}(I-\rho^{2})^{-1}\rho,\quad i=1,\dots, N-1,
\ j=2,\dots, N\nonumber\\
\widetilde{C}_{i,i}&=&
R_{0}^{-1}(I-\rho^{2})^{-1}(I+\rho^{2}),\quad i=2,\dots, N-1.
 \label{repmatfv}
 \end{eqnarray}

The following additional assumption is now  considered:

\medskip
 
 \noindent \textbf{Assumption A3.} The eigenvectors $\{ \psi_{k}\}_{k\geq 1}$ of $\rho$ satisfy  $\{ \psi_{k}\}_{k\geq 1}\subset R_{0}^{1/2}(H).$ 
 
\medskip

Under \textbf{Assumption A3}, the next lemma provides the series expansion of the functional entries of  $\mathbf{C}^{-1},$ leading to  the derivation below of the generalized least-squared estimator $\widehat{\boldsymbol{\beta}}_{N}$ of $\boldsymbol{\beta},$ under \textbf{Assumption A4}.
 \begin{lemma}
Under  \textbf{Assumption A3}, since $\psi_{k}\in \rho(H),$ for every $k\geq 1,$  the functional entries of matrix operator in (\ref{repmatfv}) admit the following series expansion in the norm of $H:$
 
 \begin{eqnarray}
\widetilde{C}_{1,1}(f)&=&\widetilde{C}_{N,N}(f)= R_{0}^{-1}(I-\rho^{2})^{-1}(f)\nonumber\\ &=& \sum_{k,l}
 \frac{1}{1-\lambda_{k}^{2}(\rho)}R_{0}^{-1}(\psi_{k})(\psi_{l})\left\langle \psi_{k},f\right\rangle_{H}\psi_{l}\nonumber\\
 &=& \sum_{k,l}a_{l,k}\left\langle \psi_{k},f\right\rangle_{H}\psi_{l} ,\quad \forall f\in H\nonumber\\ \widetilde{C}_{i,i+1}(f)&=&\widetilde{C}_{j,j-1}(f)= -R_{0}^{-1}(I-\rho^{2})^{-1}\rho (f) 
\nonumber\\
&=& -\sum_{k,l}\frac{\lambda_{k}(\rho)}{1-\lambda_{k}^{2}(\rho)}
 R_{0}^{-1}(\psi_{k})(\psi_{l})\left\langle \psi_{k},f\right\rangle_{H}\psi_{l}\nonumber \\
 &=& \sum_{k,l}b_{l,k}\left\langle \psi_{k},f\right\rangle_{H}\psi_{l} ,\ \forall f\in H,\ i=1,\dots, N-1,
\ j=2,\dots, N\nonumber\\
\widetilde{C}_{i,i}(f)&=& R_{0}^{-1}(I-\rho^{2})^{-1}(I+\rho^{2})(f) =\sum_{k,l}
 \frac{1+\lambda_{k}^{2}(\rho)}{1-\lambda_{k}^{2}(\rho)}
 R_{0}^{-1}(\psi_{k})(\psi_{l})
  \left\langle \psi_{k},f\right\rangle_{H}\psi_{l}\nonumber\\ 
 &=& \sum_{k,l}c_{l,k}\left\langle \psi_{k},f\right\rangle_{H}\psi_{l} ,\quad \forall f\in H,\quad i=2,\dots, N-1.
   \label{sefecinh}
 \end{eqnarray}
 \end{lemma}
 
 The proof follows  from \textbf{Assumption A3}, and Spectral Theorems for compact self-adjoint operators  (see Dautray and Lions 1985, pp. 112-126).
 
\medskip
 
 From (\ref{invmatcovop})--(\ref{sefecinh}),  
 $\mathbf{C}^{-1}$  then admits the following series representation: For every $\mathbf{f}=(f_{1},
 \dots,f_{N})^{T},$  $\mathbf{g}=(g_{1},\dots,g_{N})^{T}\in H^{N},$
\begin{eqnarray}
\mathbf{C}^{-1}(\mathbf{f})(\mathbf{g}) &=& \sum_{k,l}[\boldsymbol{\Psi}_{l}^{\star}(\mathbf{g})]^{T}
\mathbf{H}_{l,k}\boldsymbol{\Psi}_{k}^{\star}
(\mathbf{f})\label{eqmcoCinv}\\
\mathbf{H}_{l,k}&:=&\left[\begin{array}{ccccc} a_{l,k} & b_{l,k}  & 0 & \ldots   & 0  \\
b_{l,k}  &   c_{l,k} &  b_{l,k}&  \ldots  &  0 \\
 \vdots & \ddots &  \ddots  & \ddots & \vdots\\
 0& \ldots  & b_{l,k} &   c_{l,k} & b_{l,k} \\
0 &0 &  \ldots &   b_{l,k}  & 
a_{l,k}\\
 \end{array}\right],\label{sdcinvcoo}
\end{eqnarray}
\noindent where $a_{l,k},b_{l,k},c_{l,k},$ $k,l\geq 1,$ have been introduced in (\ref{sefecinh}).
Equations  (\ref{eqmcoCinv})--(\ref{sdcinvcoo}) define the norm in the RKHS $\mathcal{H}(\boldsymbol{\varepsilon})$  of $\boldsymbol{\varepsilon}=(\varepsilon_{1},\dots, \varepsilon_{N})^{T},$
given by
 \begin{equation}
\|\mathbf{f}\|_{\mathcal{H}(\boldsymbol{\varepsilon})}^{2}=\mathbf{C}^{-1}(\mathbf{f})(\mathbf{f})=
\sum_{k,l}[\boldsymbol{\Psi}_{l}^{\star}(\mathbf{f})]^{T}
\mathbf{H}_{l,k}\boldsymbol{\Psi}_{k}^{\star}(\mathbf{f}),\quad \forall \mathbf{f}\in \mathcal{H}(\boldsymbol{\varepsilon}).
\label{sdcinvcoo2}
 \end{equation}
  \section{Functional regression parameter estimation}
\label{sec3b}
From a functional sample  $Y_{1},\dots ,Y_{N},$ the functional parameter vector $\boldsymbol{\beta}$ in    (\ref{modelreg}) is estimated  by applying generalized least-squared. Thus, considering  (\ref{sdcinvcoo2}), this estimator is computed as the solution to the  minimization problem
\begin{eqnarray} \widehat{\boldsymbol{\beta }}_{N}&:=&\min_{\boldsymbol{\beta}\in H^{p}} L^{2}(\boldsymbol{\beta})=\min_{\boldsymbol{\beta}\in H^{p}}\|\mathbf{Y}-\mathbf{X}(\boldsymbol{\beta })\|_{\mathcal{H}(\boldsymbol{\varepsilon})}^{2}\nonumber \\
&=&\min_{\boldsymbol{\beta}\in H^{p}}(\mathbf{Y}-\mathbf{X}(\boldsymbol{\beta }))^{T}\mathbf{C}^{-1}(\mathbf{Y}-\mathbf{X}(\boldsymbol{\beta }))\nonumber \\
&=&\min_{\boldsymbol{\beta}\in H^{p}}\sum_{k,l}[\boldsymbol{\Psi}_{l}^{\star}(\mathbf{Y}-\mathbf{X}(\boldsymbol{\beta })]^{T}
\mathbf{H}_{l,k}\boldsymbol{\Psi}_{k}^{\star}(\mathbf{Y}-\mathbf{X}(\boldsymbol{\beta })),
\label{lossfunction}
\end{eqnarray}
\noindent where \begin{eqnarray}\mathbf{X}&:=&\left[\begin{array}{c}\mathbf{X}_{1}^{T}\\
\vdots\\
\mathbf{X}_{N}^{T}\\\end{array}\right]=\left[\begin{array}{ccc}X_{1}^{1}&\dots & X_{1}^{p}\\
\vdots &\vdots &\vdots\\
X_{N}^{1}& \dots &X_{N}^{p}\\
\end{array}\right]=\left[\mathbf{X}^{1},\dots,  \mathbf{X}^{p}\right]
\nonumber\\
\mathbf{X}_{i}^{T}&:=&(X_{i}^{1},\dots , X_{i}^{p}),\quad  i=1,\dots,N,\nonumber\\
\mathbf{X}^{j}&=&(X^{j}_{1},\dots,X^{j}_{N})^{T},\quad j=1,\dots,p
\label{dmfcr} \end{eqnarray}\begin{eqnarray}
X_{n}^{i}(f)(g)&=&\sum_{k,l}x^{i}_{k,l}(n)\left\langle f,\psi_{l}\right\rangle_{H}\left\langle g,\psi_{k}\right\rangle_{H},\nonumber\\ &&  \forall f,g\in H,\quad i=1,\dots,p,\ n=1,\dots,N  \label{regressors2b}\\
\mathbf{Y} &:=&\left( \mathbf{Y}_{1},\dots,  \mathbf{Y}_{N}\right)^{T}
\quad \boldsymbol{\beta} = \left( \boldsymbol{\beta}_{1},\dots,  \boldsymbol{\beta}_{p}\right)^{T}.
\label{elemenmp}
\end{eqnarray}

Consider \begin{eqnarray}\boldsymbol{\beta}&=&\left( \sum_{k\geq 1}\left\langle \beta_{1},\psi_{k}\right\rangle_{H}\psi_{k},\dots,\sum_{k\geq 1}\left\langle \beta_{p},\psi_{k}\right\rangle_{H}\psi_{k}\right)^{T}\nonumber\\&=&\left( \sum_{k\geq 1}\beta_{1k}\psi_{k},\dots,\sum_{k\geq 1}\beta_{pk}\psi_{k}\right)^{T},\nonumber\end{eqnarray} \noindent and the following assumption: 

\medskip

\noindent \textbf{Assumption A4.}
Assume the regularity conditions  ensuring the following identities  hold:
\begin{eqnarray}
\frac{\partial \boldsymbol{\Psi}_{k}^{\star}
\mathbf{X}(\boldsymbol{\beta })}{\partial \beta_{j_{0}h_{0}}}&=&
\left(\sum_{j=1}^{p} \sum_{h=1}^{\infty }\frac{\partial
 x^{j}_{k,h}(1)\beta_{jh}}{\partial \beta_{j_{0}h_{0}}},\dots,
 \sum_{j=1}^{p} \sum_{h=1}^{\infty }\frac{\partial
 x^{j}_{k,h}(N)\beta_{jh}}{\partial \beta_{j_{0}h_{0}}}\right)^{T}
 \nonumber\\ &=&\left(
 x^{j_{0}}_{k,h_{0}}(1),\dots, x^{j_{0}}_{k,h_{0}}(N) \right)^{T},\label{a4}
\end{eqnarray}
\noindent  with uniform convergence with respect to $k\geq 1,$  for  $j_{0}=1,\dots,p,$  and $h_{0}\geq 1.$

\medskip

\noindent Under \textbf{Assumption A4}, denote, for $j_{0}=1,\dots,p,$
 \begin{eqnarray}\frac{\partial \boldsymbol{\Psi}_{k}^{\star}
\mathbf{X}(\boldsymbol{\beta })}{\partial \boldsymbol{\beta}_{j_{0}}}&=&
\left(\left(\sum_{j=1}^{p} \sum_{h=1}^{\infty }\frac{\partial
 x^{j}_{k,h}(1)\beta_{jh}}{\partial \beta_{j_{0}h_{0}}}\right)_{h_{0}\geq 1},\dots,
 \left(\sum_{j=1}^{p} \sum_{h=1}^{\infty }\frac{\partial
 x^{j}_{k,h}(N)\beta_{jh}}{\partial \beta_{j_{0}h_{0}}}\right)_{h_{0}\geq 1}\right)^{T}\nonumber\\
 &=&\left(
 \left(x^{j_{0}}_{k,h_{0}}(1)\right)_{h_{0}\geq 1},\dots, \left(x^{j_{0}}_{k,h_{0}}(N)\right)_{h_{0}\geq 1} \right)^{T}\equiv  
\boldsymbol{\Psi}_{k}^{\star}\mathbf{X}^{j_{0}},\label{partialb}
\end{eqnarray}
\noindent  where $\mathbf{X}^{j_{0}}$ has been introduced in equations 
(\ref{dmfcr})--(\ref{regressors2b}), and $\equiv$ denotes the identification  $[l^{2}]^{N}\equiv H^{N}$ established   by the isometry defined in terms of the orthonormal basis $\{\psi_{k}\}_{k\geq 1}.$
 Then, under \textbf{Assumption A4},  from equations (\ref{lossfunction})--(\ref{partialb}),  for each $j_{0}=1,\dots,p,$  

 \begin{eqnarray}&&\frac{\partial \|\mathbf{Y}-\mathbf{X}(\boldsymbol{\beta }) \|^{2}_{\mathcal{H}(\boldsymbol{\varepsilon})}}{\partial \boldsymbol{\beta}_{j_{0}}}= \sum_{k,l}\frac{\partial [\boldsymbol{\Psi}_{l}^{\star}(\mathbf{Y}-\mathbf{X}(\boldsymbol{\beta })]^{T}
\mathbf{H}_{l,k}\boldsymbol{\Psi}_{k}^{\star}(\mathbf{Y}-\mathbf{X}(\boldsymbol{\beta }))}{\partial \boldsymbol{\beta }_{j_{0}}}
\nonumber\\
&&=-\sum_{k,l}\left[\mathbf{X}^{j_{0}}\right]^{T}\boldsymbol{\Psi}_{l}\mathbf{H}_{l,k}\boldsymbol{\Psi}_{k}^{\star}(\mathbf{Y}-\mathbf{X}(\boldsymbol{\beta }))+[\boldsymbol{\Psi}_{l}^{\star}(\mathbf{Y}-\mathbf{X}(\boldsymbol{\beta })]^{T}\mathbf{H}_{l,k}\boldsymbol{\Psi}_{k}^{\star}\mathbf{X}^{j_{0}}.\nonumber\\
\label{eqminimnorm}
\end{eqnarray} 
From (\ref{eqminimnorm}), the minimizer of  (\ref{lossfunction})
  with respect to $\boldsymbol{\beta},$ i.e., the generalized least-squared estimator $ \widehat{\boldsymbol{\beta}}_{N}$ of $\boldsymbol{\beta}$ is  given by the solution to the following matrix functional equation:
\begin{eqnarray}
&&-\mathbf{X}^{T}\mathbf{C}^{-1}(\mathbf{Y}-\mathbf{X}(\boldsymbol{\beta }))=\mathbf{X}^{T}\mathbf{C}^{-1}\mathbf{X}(\boldsymbol{\beta })-\mathbf{X}^{T}\mathbf{C}^{-1}\mathbf{Y}
=\mathbf{0}\nonumber\\
&&-(\mathbf{Y}-\mathbf{X}(\boldsymbol{\beta }))^{T}\mathbf{C}^{-1}\mathbf{X}=\boldsymbol{\beta }^{T}\mathbf{X}^{T}\mathbf{C}^{-1}\mathbf{X}-\mathbf{Y}^{T}\mathbf{C}^{-1}\mathbf{X}
=\mathbf{0}.
\label{glse2}
\end{eqnarray}

\noindent  Furthermore, under the condition that the inverse 
$(\mathbf{X}^{T}\mathbf{C}^{-1}\mathbf{X})^{-1}$ exists, the solution to (\ref{glse2}) is defined as
\begin{eqnarray}\widehat{\boldsymbol{\beta}}_{N}&=&\left(\mathbf{X}^{T}\mathbf{C}^{-1}\mathbf{X}\right)^{-1}\mathbf{X}^{T}
\mathbf{C}^{-1}(\mathbf{Y}_{N})\nonumber\\
&=&\boldsymbol{\beta}+\left(\mathbf{X}^{T}\mathbf{C}^{-1}\mathbf{X}\right)^{-1}\mathbf{X}^{T}
\mathbf{C}^{-1}(\boldsymbol{\varepsilon}_{N}).\label{glse}
\end{eqnarray}
Then, from  (\ref{glse}), 

\begin{eqnarray}
E[\widehat{\boldsymbol{\beta}}_{N}]&=& \boldsymbol{\beta},\quad 
E[(\widehat{\boldsymbol{\beta}}_{N}-\boldsymbol{\beta})(\widehat{\boldsymbol{\beta}}_{N}-\boldsymbol{\beta})^{T}]=
 \left(\mathbf{X}^{T}\mathbf{C}^{-1}\mathbf{X}\right)^{-1}\nonumber\\
 \widehat{\boldsymbol{\beta}}_{N}\in H^{p} &\Leftrightarrow &\boldsymbol{\varepsilon}^{T}\mathbf{C}^{-1}\mathbf{X} \left(\mathbf{X}^{T}\mathbf{C}^{-1}\mathbf{X}\right)^{-1}\left(\mathbf{X}^{T}
\mathbf{C}^{-1}\mathbf{X}\right)^{-1}\mathbf{X}^{T}\mathbf{C}^{-1}\boldsymbol{\varepsilon}<\infty,\quad \mbox{a.s.},\nonumber\\
\label{varest}
\end{eqnarray}
\noindent where a.s. denotes the almost surely equality, and the last condition in  (\ref{varest}) should be assumed for the suitable definition of the parameter estimator $\widehat{\boldsymbol{\beta}}_{N}.$
  
  \subsection{Asymptotic normality}
  From (\ref{varest}), applying Theorem 2.7 in Bosq (2000), the following central limit result provides the asymptotic normal distribution of the generalized least-squared estimator $\widehat{\boldsymbol{\beta}}_{N},$ as $N\to \infty.$
  \begin{theorem}
  \label{CLRGLSE}
 Under \textbf{Assumptions A1--A4}, let $\widehat{\boldsymbol{\beta}}_{N}$ be the generalized 
  least-squared estimator defined in (\ref{glse}) satisfying (\ref{varest}).
Assume that    $\{\delta_{n},\ n\in \mathbb{Z}\}$ is  Gaussian 
strong-white noise in $H.$
   Then, as $N\to \infty,$
  $$\frac{\left(\mathbf{X}^{T}\mathbf{C}^{-1}\mathbf{X}\right)^{1/2}\left(\widehat{\boldsymbol{\beta}}_{N}-\boldsymbol{\beta}\right)}{\sqrt{N}}\to_{D} \mathcal{N}\left(\mathbf{0},\mathbf{I}_{N\times N}
  \right),$$ \noindent where $\mathbf{I}_{N\times N}$ denotes the identity operator on $H^{N}.$
  
 \end{theorem}
 \begin{proof}
  The proof directly follows from  
  Theorem 2.7 in Bosq (2000), since, from equation (\ref{varest}), the $H$-valued components of the functional vector
 \begin{equation*} 
  \boldsymbol{\mathcal{Z}}=\left(\begin{array}{c}\mathcal{Z}_{1}\\
  \vdots\\ \mathcal{Z}_{N}\\ \end{array}\right)=\left(\mathbf{X}^{T}\mathbf{C}^{-1}\mathbf{X}\right)^{1/2}\left(\widehat{\boldsymbol{\beta}}_{N}-\boldsymbol{\beta}\right)\label{e1th1}\end{equation*}
  \noindent  are independent and identically distributed $H$-valued random variables, with, for $i=1,\dots,p,$  
  \begin{equation} \mathcal{Z}_{i}=\sum_{j=1}^{N}B_{i,j}(\varepsilon_{j})\sim \mathcal{N}\left(\mathbf{0}, I\right),\label{e2th1}\end{equation}
  \noindent and, for $j=1,\dots,N,$ $B_{i,j}$ denotes the $(i,j)$ functional entry of \linebreak $\left(\mathbf{X}^{T}\mathbf{C}^{-1}\mathbf{X}\right)^{1/2}\left(\mathbf{X}^{T}\mathbf{C}^{-1}\mathbf{X}\right)^{-1}\mathbf{X}^{T}
\mathbf{C}^{-1}.$ As before,  $I$ denotes the identity operator on $H.$ The Central Limit Result provided in Theorem 2.7 in Bosq (2000) for i.i.d. $H$-valued random variables then lead to the desired result.

  \end{proof}
\subsection{Strong consistency}

The following  conditions are required:

\medskip

\noindent \textbf{Assumption A5.} There exists $Q\in \mathcal{L}(H^{p})$ such that 
\begin{eqnarray}&& \left\|\left(\frac{\mathbf{X}^{T}
\mathbf{C}^{-1}\mathbf{X}}{N}\right)^{-1}-\mathbf{Q}\right\|_{\mathcal{L}(H^{p})}\to 0,
\quad  N\to \infty,\label{detlimits0}
\end{eqnarray} where   $\mathcal{L}(H^{p})$ denotes the space of bounded linear operators on $H^{p}.$ 

\medskip

\noindent \textbf{Assumption A6}. For every $N\geq 2,$ $\mathbf{X}$ is such that   
$\mathbf{C}^{-1}\mathbf{X}\mathbf{X}^{T}\mathbf{C}^{-1}\in $\linebreak$\mathcal{L}(H^{N}),$ with  $\mathcal{L}(H^{N})$ denoting the space of bounded linear operators on $H^{N}.$
\begin{theorem}
\label{sc}
Under  \textbf{Assumptions A1--A6}, the generalized least-squared estimator $\widehat{\boldsymbol{\beta }}_{N}$ satisfying 
(\ref{glse})--(\ref{varest}) is strong consistent in $H^{p},$ i.e.,
\begin{equation}
\|\widehat{\boldsymbol{\beta }}_{N}-\boldsymbol{\beta }\|_{H^{p}}\to_{a.s.} 0,\quad N\to \infty.\label{escth}\end{equation}
\end{theorem}
From (\ref{glse}) and (\ref{detlimits0}),   
 as $N\to \infty$:
\begin{eqnarray}
&&\left\|\widehat{\boldsymbol{\beta}}_{N}-\boldsymbol{\beta}\right\|_{H^{p}}^{2}\leq 
\left\|\left(\frac{\mathbf{X}^{T}\mathbf{C}^{-1}\mathbf{X}}{N}\right)^{-1}\right\|_{\mathcal{L}(H^{p})}^{2}\left\|\frac{\mathbf{X}^{T}\mathbf{C}^{-1}(\boldsymbol{\varepsilon})}{N}\right\|_{H^{p}}^{2},\quad \mbox{a.s.}
\label{lim2hh}
\end{eqnarray}

Furthermore, applying Cauchy–-Schwarz inequality, the following a.s. identities hold:
\begin{eqnarray}
\left\|\frac{\mathbf{X}^{T}\mathbf{C}^{-1}(\boldsymbol{\varepsilon})}{N}\right\|_{H^{p}}^{2}&=&\left\langle \frac{\mathbf{X}^{T}\mathbf{C}^{-1}(\boldsymbol{\varepsilon})}{N},\frac{\mathbf{X}^{T}\mathbf{C}^{-1}(\boldsymbol{\varepsilon})}{N}\right\rangle_{H^{p}}\nonumber\\
&=& \frac{1}{N^{2}}\left\langle \mathbf{C}^{-1}\mathbf{X}\mathbf{X}^{T}\mathbf{C}^{-1}(\boldsymbol{\varepsilon}),
\boldsymbol{\varepsilon}\right\rangle_{H^{N}}
\nonumber\\
&\leq&
\frac{1}{N^{2}}\left\|\mathbf{C}^{-1}\mathbf{X}\mathbf{X}^{T}\mathbf{C}^{-1}(\boldsymbol{\varepsilon})\right\|_{H^{N}}\left\|\boldsymbol{\varepsilon}\right\|_{H^{N}}\nonumber\\
&\leq &\frac{1}{N^{2}}\left\|\mathbf{C}^{-1}\mathbf{X}\mathbf{X}^{T}\mathbf{C}^{-1}
\right\|_{\mathcal{L}(H^{N})}\left\|\boldsymbol{\varepsilon}\right\|_{H^{N}}^{2}.
\label{lim2b}
\end{eqnarray}
Now, consider
\begin{equation}
E\left\|
\boldsymbol{\varepsilon}\right\|_{H^{N}}^{2}=\sum_{j=1}^{N}E\|\varepsilon_{j}\|_{H}^{2}=N\|R_{0}\|_{\mathcal{N}(H)},
\label{eqdemosc}
\end{equation}
\noindent where $\|\cdot\|_{\mathcal{N}(H)}$ denotes the nuclear or trace operator norm.
From (\ref{eqdemosc}),
\begin{equation}
\frac{\left\|
\boldsymbol{\varepsilon}\right\|_{H^{N}}^{2}}{N^{2}}\to_{a.s. } 0,\quad N\to \infty.
\label{opnormtrace}
\end{equation}

From equations (\ref{lim2b}) and (\ref{opnormtrace}), under \textbf{Assumption A6},
\begin{equation}
\left\|\frac{\mathbf{X}^{T}\mathbf{C}^{-1}(\boldsymbol{\varepsilon})}{N}\right\|_{H^{p}}^{2}\to_{a.s.} 0,\quad N\to \infty .
\label{limitas}
\end{equation}

Under \textbf{Assumption A5}, from equations (\ref{lim2hh}) and (\ref{limitas}), the strong-consistency in $H^{p}$ of 
 $\widehat{\boldsymbol{\beta}}_{N}$ holds.
\section{Practical implementation}

\label{esresiduals}

In practice, when $R_{0}$ and $R_{1}$ are unknown,   ordinary 
least-squared is first applied,  that is, 
$\widetilde{\boldsymbol{\beta }}_{N}=\left(\mathbf{X}^{T}
\mathbf{X}\right)^{-1}\mathbf{X}^{T}(\mathbf{Y}),$
 is  computed, and 
 the resulting residuals are used to approximate $R_{0}$ and $R_{1}$   as follows:
 \begin{eqnarray}
\widetilde{R}_{0}^{N}&:=&\frac{1}{N}\sum_{m=1}^{N}[\mathbf{Y}_{m}-
\mathbf{X}_{m}^{T}(\widetilde{\boldsymbol{\beta }}_{N})]\otimes [\mathbf{Y}_{m}-
\mathbf{X}_{m}^{T}(\widetilde{\boldsymbol{\beta }}_{N})]\nonumber\\
\widetilde{R}_{1}^{N-1}&:=&\frac{1}{N-1}\sum_{m=1}^{N-1}[\mathbf{Y}_{m}-
\mathbf{X}_{m}^{T}(\widetilde{\boldsymbol{\beta }}_{N})]\otimes [\mathbf{Y}_{m+1}-
\mathbf{X}_{m+1}^{T}(\widetilde{\boldsymbol{\beta }}_{N})].
\label{estste1}
\end{eqnarray}
In a second step, these empirical covariance operators  are considered in the computation of equation  (\ref{glse}),  in terms of a suitable orthonormal  empirical basis 
 of $H.$ Consider, in particular, 
$\{\phi_{jN}\}_{j\geq 1}$  the system of   eigenvectors of the empirical autocovariance operator $\widetilde{R}_{0}^{N},$ satisfying (see Bosq, 2000, pp. 102--103)
\begin{eqnarray}
\widetilde{R}_{0}^{N}\phi_{jN}&=& \lambda_{jN}\phi_{jN},\ j\geq 1,
\nonumber\\
\lambda_{1N}&\geq &\dots\geq \lambda_{NN}\geq 0=\lambda_{N,N+1}=\lambda_{N,N+2},\dots,  \label{ev1}
\end{eqnarray}

\noindent where  $\{\lambda_{jN}\}_{j\geq 1}$  is the system of  eigenvalues  of  $\widetilde{R}_{0}^{N}.$ The operators $[\widetilde{R}_{0}^{N}]^{-1}$ and $\widehat{\widetilde{\rho}}_{N}=\widetilde{R}_{1}^{N-1}[\widetilde{R}_{0}^{N}]^{-1}$ can then be computed in terms of such empirical eigenvectors. Thus,  
 the  $H$-valued residuals 
\begin{equation}
\widetilde{\varepsilon}_{n}:= Y_{n}-\widetilde{Y}_{n}=Y_{n}-X_{n}^{1}(\widetilde{\boldsymbol{\beta }}_{1}^{N}) -\dots - X_{n}^{p}(\widetilde{\boldsymbol{\beta }}_{p}^{N}),\quad n=1,\dots,N,
\label{fs}
\end{equation}

\noindent of the ordinary least-squared estimator  $\widetilde{\boldsymbol{\beta }}_{N}=(
\widetilde{\boldsymbol{\beta }}_{1}^{N},\dots,\widetilde{\boldsymbol{\beta }}^{N}_{p})^{T}$ 
 are considered, in the computation of the following estimator of the autocorrelation operator of the error process: 
   \begin{eqnarray}
 &&\hspace*{-0.3cm} \widehat{\widetilde{\rho}}_{k_{N}}:=\sum_{i=1}^{k_{N}}\sum_{j=1}^{k_{N}}\widehat{\widetilde{\rho}}_{i,j,N}\phi_{iN}\otimes
\phi_{jN};\ 
  \widehat{\widetilde{\rho}}_{i,j,N}=\frac{1}{N-1}\sum_{
n=1}^{N-1}\left\langle \widetilde{\varepsilon}_{n},\phi_{iN}\right\rangle_{H}
\frac{\left\langle \widetilde{\varepsilon}_{n+1},\phi_{jN}\right\rangle_{H}}{\lambda_{jN}}.\nonumber\\\label{estcomponent}
\end{eqnarray}
\noindent   Here, $k_{N}$ denotes the truncation parameter, with  $k_{N}\leq N,$ $k_{N}\to \infty,$  and $\frac{k_{N}}{N}\to 0,$ $N\to \infty $ 
(see  Bosq, 2000).
 The estimator (\ref{estcomponent})
 has 
the same asymptotic properties as the estimator  of $\rho,$ computed from 
  $\{\varepsilon_{n},\ n=1,\dots,N\},$ in the case where  the ordinary least-squared estimator $\widetilde{\boldsymbol{\beta }}_{N}$ of $\boldsymbol{\beta}$
is strong consistent in $H^{p}.$ 
In particular, $\widehat{\widetilde{\rho}}_{k_{N}}$ is also 
strong consistent, in the norm of $\mathcal{L}(H)$ (see  Chapter 8 in Bosq, 2000). Note that   
\begin{eqnarray}
 \widetilde{\varepsilon}_{n}&=& Y_{n}-
X_{n}^{1}(\widetilde{\boldsymbol{\beta }}_{1}^{N}) -\dots - X_{n}^{p}(\widetilde{\boldsymbol{\beta }}_{p}^{N})\nonumber\\
&=& \varepsilon_{n}+X_{n}^{1}\left(\boldsymbol{
\beta}_{1}-\widetilde{\boldsymbol{\beta }}_{1}^{N}\right)+\dots +
X_{n}^{p}\left(\boldsymbol{
\beta}_{p}-\widetilde{\boldsymbol{\beta }}_{p}^{N}\right)\nonumber\\
&=&\varepsilon_{n}+o_{a.s.}(1),\quad N\to \infty,\nonumber
\label{consrho}
\end{eqnarray}
\noindent in view of the strong consistency of $\widetilde{\boldsymbol{\beta }}_{N},$ leading to  
\begin{eqnarray}\widetilde{R}_{0}^{N}&=&
\frac{1}{N}\sum_{n=1}^{N}\widetilde{\varepsilon}_{n}\otimes 
\widetilde{\varepsilon}_{n}=\frac{1}{N}\sum_{n=1}^{N}\varepsilon_{n}\otimes 
\varepsilon_{n}+o_{a.s.}(1)=R_{0}^{N}+o_{a.s.}(1)\nonumber\\
\widetilde{R}_{1}^{N-1}&=&\frac{1}{N-1}\sum_{n=1}^{N-1}\widetilde{\varepsilon}_{n}\otimes 
\widetilde{\varepsilon}_{n+1}=\frac{1}{N-1}\sum_{n=1}^{N-1}\varepsilon_{n}\otimes 
\varepsilon_{n+1}+o_{a.s.}(1)
\nonumber\\
&=&R_{1}^{N-1}+o_{a.s.}(1),\label{atucovvarep}
\end{eqnarray}
\noindent which also implies the 
strong consistency of $\widetilde{R}_{0}^{N},$ and  $\widetilde{R}_{1}^{N-1},$ involved in the computation of (\ref{glse}), when $R_{0}$ and $R_{1}$ are unknown. 
For the strong consistency of the ordinary least-squared parameter estimator $\widetilde{\boldsymbol{\beta }}_{N},$ under dependent errors, the following sufficient conditions are assumed:

\medskip

\noindent \textbf{Assumption $\widetilde{A5}$}. There exists $\widetilde{Q}\in \mathcal{L}(H^{p})$ such that 
\begin{eqnarray}&&\left\| \left(\frac{\mathbf{X}^{T}
\mathbf{X}}{N}\right)^{-1}-\widetilde{Q}\right\|_{\mathcal{L}(H^{p})}\to 0,\quad N\to \infty.
 \label{detlimits0bb}
\end{eqnarray} 

\medskip

\noindent \textbf{Assumption $\widetilde{A6}$}. $\mathbf{X}$ is such that  $\mathbf{X}\mathbf{X}^{T}\in \mathcal{L}(H^{N}),$ for every $N\geq 2.$

\begin{proposition}
\label{prole}
Under \textbf{Assumptions $\widetilde{A5}$--$\widetilde{A6}$}, the ordinary least-squared parameter estimator $\widetilde{\boldsymbol{\beta }}_{N}$ is strong consistent.
\end{proposition}

Under \textbf{Assumptions $\widetilde{A5}$--$\widetilde{A6}$},   the proof  of Proposition \ref{prole} is derived, in a similar way to Theorem \ref{sc}, from the following a.s. inequality: 
\begin{eqnarray}
&&\|\widetilde{\boldsymbol{\beta}}_{N}-\boldsymbol{\beta}\|_{H^{p}}^{2}\leq 
\left\|\left(\frac{\mathbf{X}^{T}\mathbf{X}}{N}\right)^{-1}
\right\|_{\mathcal{L}(H^{p})}^{2}\left\|\frac{\mathbf{X}^{T}(\boldsymbol{\varepsilon})}{N}\right\|_{H^{p}}^{2}
\label{lim2}
\end{eqnarray}

\begin{remark}
\label{rem4}
When $R_{0}$ and $R_{1}$ are unknown,  the functional entries  $\widetilde{C}_{ij},$ $i,j=1,\dots,N,$ of $\mathbf{C}^{-1}=\left(\widetilde{C}_{ij}\right)_{i,j=1\dots,N}$ in (\ref{repmatfv}) can be replaced by $\widetilde{R}_{0}^{N}$ and 
$\widehat{\widetilde{\rho}}_{k_{N}}=F(\widetilde{R}_{1}^{N-1},\widetilde{R}_{0}^{N})=\pi_{k_{N}}^{\star}\widetilde{R}_{1}^{N-1}[\widetilde{R}_{0}^{N}]^{-1}\pi_{k_{N}}$ (see equations (\ref{estste1})--(\ref{estcomponent})). Here, $\pi_{k_{N}}$ denotes the orthogonal projector into the subspace of $H$ generated by the  eigenvectors 
$\{\phi_{jN},\ j=1,\dots,k_{N}\}$ of $\widetilde{R}_{0}^{N},$ with, as before,  $k_{N}\leq N,$ \linebreak $k_{N}\to \infty,$  and $\frac{k_{N}}{N}\to 0,$ $N\to \infty .$ \textbf{Assumptions $\widetilde{A5}$--$\widetilde{A6}$} ensure the strong consistency of the ordinary least squared estimator $\widetilde{\boldsymbol{\beta}}_{N}$ of $\boldsymbol{\beta }.$ From equation (\ref{atucovvarep}),  the  eigenvectors 
$\{\phi_{jN},\ j=1,\dots,k_{N}\}$ of $\widetilde{R}_{0}^{N}$ a.s.   converge  to the eigenvectors of $\widehat{R}_{0}^{N}=\frac{1}{N}\sum_{n=1}^{N}\varepsilon_{n}\otimes 
\varepsilon_{n},$  as $N\to \infty,$ since $\widetilde{R}_{0}^{N}\to_{a.s.} \widehat{R}_{0}^{N},$ $N\to \infty.$ (Also $\widetilde{R}_{1}^{N-1}\to_{a.s.} \widehat{R}_{1}^{N-1},$ $N\to \infty $).
 Under the conditions of  Theorem 8.8  in Bosq (2000) (see Section 8.3 in Bosq, 2000),   the  strong consistency of $\widehat{\widetilde{\rho}}_{k_{N}}$ then  holds,   when $\rho$ is a Hilbert--Schmidt operator, considering $k_{N}$ such that      \begin{equation}\frac{N \lambda_{k_{N}}^{2}(R_{0})}{\left(\sum_{j=1}^{k_{N}}a_{j}\right)^{2}\log(N)}\to \infty,\quad \quad N\to \infty,\label{eqtrucpar}
\end{equation}
\noindent where \begin{eqnarray}a_{1}&=& 2\sqrt{2}(\lambda_{1}(R_{0})-\lambda_{2}(R_{0}))^{-1}\nonumber\\ a_{j}&=& 2\sqrt{2}\max\left[(\lambda_{j-1}(R_{0})-\lambda_{j}(R_{0}))^{-1},(\lambda_{j}(R_{0})-\lambda_{j+1}(R_{0}))^{-1}\right], \ j\geq 2.\nonumber
\end{eqnarray}

\noindent Thus, the strong consistency of the corresponding plug-in generalized least-squared estimator, $\widehat{\widetilde{\boldsymbol{\beta}}}_{N},$   holds from 
the strong consistency of $\widetilde{\boldsymbol{\beta}}_{N},$ under the conditions of  Theorem 8.8  in Bosq (2000).

 \end{remark}
\subsubsection*{ARH(1)-based  estimation  of the functional response}
 
The following estimator of the $H$-valued dynamical response is considered: 
\begin{equation}
\widehat{Y}_{N}:=X_{N}^{1}(\widehat{\boldsymbol{\beta }}_{1}^{N}) +\dots + X_{N}^{p}(\widehat{\boldsymbol{\beta }}_{p}^{N})+\widehat{\widetilde{\rho}}_{k_{N}}(\widehat{\varepsilon}_{N-1}),
\label{eqfer}
\end{equation}
 \noindent where  $\widehat{\widetilde{\rho}}_{k_{N}}(\widehat{\varepsilon}_{N-1})$    is  computed in a similar way to  (\ref{estcomponent}), from the residuals $\widehat{\varepsilon}_{n}=Y_{n}-X_{n}^{1}(\widehat{\boldsymbol{\beta }}_{1}^{N})-\dots - X_{n}^{p}(\widehat{\boldsymbol{\beta }}_{p}^{N}),$ $n=1,\dots,N,$ with $\widehat{\boldsymbol{\beta }}_{i}^{N},$ $i=1,\dots,p,$ being the generalized least-squared estimators of the components of $\boldsymbol{\beta},$ based on the observation of $\mathbf{Y}_{1},\dots, \mathbf{Y}_{N},$ computed in terms of $\mathbf{C}^{-1},$ or its empirical version, as given before, in the case where $R_{0}$ and $R_{1}$ are unknown.
 
\section{Simulation Study}
\label{Secsims}
 The performance of the presented approach is studied  in the case where the eigenvectors of the autocovariance operator of the error process  are unknown, as usually it occurs in practice. Model 2 below (see also Models 3 and 4 in Supplementary Material I), also illustrates the fact that the Hilbert--Schmidt assumption on the regressors can be relaxed to the compactness condition, under diagonal spectral design.  
Let us restrict our attention to the Gaussian case, and to the real separable Hilbert space $H=L^{2}((a,b)),$   the space of square integrable functions on $(a,b),$ with $(a,b)=(0,60).$  The following systems of eigenvectors and eigenvalues are considered:
 \begin{eqnarray}
&& \phi_{j} \left( x\right) = \frac{2}{b-a} \sin \left(
\frac{\pi j x}{b-a} \right),\quad j\geq 1 \label{eigvex}\\
R_{0}(f)(x)&=&\sum_{k=1}^{\infty}\lambda_{k}(R_{0})\int_{a}^{b}\phi_{k}(x)\phi_{k}(y)f(y)dy\label{acoresp}\\
R_{\delta}(f)(x)&=&\sum_{k=1}^{\infty}\lambda_{k}(R_{\delta})\int_{a}^{b}\phi_{k}(x)\phi_{k}(y)f(y)dy
\label{acoresp2}\\
\rho(f)(x)&=&\sum_{k=1}^{\infty}\lambda_{k}(\rho)\int_{a}^{b}\phi_{k}(x)\phi_{k}(y)f(y)dy,
\label{autocorr}\\
X_{n}^{i}(\beta_{i})(x)&=&\sum_{k=1}^{\infty}x_{k}^{i}(n)\int_{a}^{b}\phi_{k}(x)\phi_{k}(y)\beta_{i}(y)dy,\quad i=1,\dots,p
\nonumber\\
\beta_{i}(x)&=&\sum_{k=1}^{\infty}\left\langle \beta_{i},\phi_{k}\right\rangle_{L^{2}((a,b))}\phi_{k}(x)=
\sum_{k=1}^{\infty}\beta_{i}(k)\phi_{k}(x),\ i=1,\dots,p.
\label{parameterest}
\end{eqnarray}
\noindent Equation (\ref{eigvex}) defines   $\{\phi_{j}\}_{j\geq 1}$ as  the eigenvectors    of the Dirichlet negative Laplacian operator on $(a,b).$  The sequences  $\{\lambda_{k}(R_{0})\}_{k\geq 1},$ $\{\lambda_{k}(R_{\delta})\}_{k\geq 1}$  and $\{\lambda_{k}(\rho)\}_{k\geq 1}$ are the  respective systems of eigenvalues of $R_{0},$  $R_{\delta}$ and $\rho.$ 
 Note that, in the examples below, $\{\psi_{k}\}_{k\geq 1}$ coincide with  the eigenvectors $\{\phi_{k}\}_{k\geq 1}$ of $R_{0}.$ Six models have been  analyzed, displaying different regularity orders. 
The observations $Y_{1},\dots, Y_{N}$ of the response  are generated from equations 
(\ref{modelreg})--(\ref{ARHerroerterm}), in terms of 
(\ref{eigvex})--(\ref{parameterest})
 (one realization of a functional sample of size $N=200$ of the response and its estimation is represented  in Supplementary Material I,  for the six Models analyzed).
 The results for the most regular and singular scenarios are displayed here, corresponding to Model 1 and 2, respectively (results on Models 3--6 are displayed, for $k_{N}=2,3,4,$ and $N=200, 600, 1000,$  in  Supplementary Material I). Tables \ref{T3a} 
and \ref{T6qq} show  the functional  empirical mean quadratic  errors 
\begin{eqnarray}
EFMQE(n)&=&\frac{1}{r}\sum_{i=1}^{r}\frac{1}{60}\sum_{x\in (0,60)}[Y^{i}_{n}(x)-\widehat{Y}^{i}_{n}(x)]^{2}, \label{eqe0}\end{eqnarray}
\noindent for the most unfavorable case, i.e., for  the largest truncation parameter value $k_{N}=4,$ and the smallest sample size $N=200.$ Here, $r$ denotes the number of repetitions generated. See also Supplementary Material I, where additional truncation parameter values and sample sizes are showed. The  CEMQEs, 
$$CEMQE(x,n)=\frac{1}{r}\sum_{i=1}^{r}[Y^{i}_{n}(x)-\widehat{Y}^{i}_{n}(x)]^{2},\ x\in (0,60),\ n=1,\dots,N,$$ \noindent are also represented, in that supplementary material. Here, $Y^{i}_{n}(x)$ denotes the value of the response at point $x\in (0,60),$  and
$\widehat{Y}^{i}_{n}(x)$ is its estimated value, for times $n=1,\dots, N=200,$ computed  from the $i$-th generation of a  functional sample of size $N,$ for $i=1,\dots,r.$ As given in Remark \ref{rem4}, the optimal $k_{N}$  is determined from the sample size, the convergence rate to zero of the empirical eigenvalues of $R_{0},$ and the distance between the empirical eigenvalues of $R_{0}.$  Indeed,   the optimal $k_{N}$ value lies in the interval $[2,4],$ for $N=200, 600, 1000$
(see  Supplementary Material I). 

Models 1 and 2 are defined from the following  parameter values: For each $k\geq 1,$ and $n\geq 1,$
\begin{eqnarray}\mbox{\textbf{Model 1}}\quad \lambda_{k}(R_{0}) &=& \frac{1}{(k+1)^{3}},\
\lambda_{k}(R_{\delta})=\frac{1}{(k+1)^{4}},\
\lambda_{k}(\rho)=\frac{1}{(k+1)}\nonumber\\
x_{k}^{1}(n)&=&\exp(-nk^{1/10}),\quad
x_{k}^{2}(n)=\exp(-nk^{15/100}),\nonumber\\
  x_{k}^{3}(n)&=&\exp(-nk^{2/10}),\quad \left\langle \beta_{1},\phi_{k}\right\rangle_{L^{2}(a,b)} =\frac{1}{(k+1)^{3/5}},\nonumber\\
\left\langle \beta_{2},\phi_{k}\right\rangle_{L^{2}(a,b)}
&=&\frac{1}{(k+1)^{7/10}},\quad\left\langle \beta_{3},\phi_{k}\right\rangle_{L^{2}(a,b)}=\frac{1}{(k+1)^{4/5}}.
\label{example3}\end{eqnarray}
\begin{eqnarray}
\mbox{\textbf{Model 2}}\quad \lambda_{k}(R_{0}) &=& \frac{1}{(k+1)^{11/10}},\
\lambda_{k}(R_{\delta})=\frac{1}{(k+1)^{12/10}}\nonumber\\ 
\lambda_{k}(\rho)&=&\frac{1}{(k+1)^{51/100}}\nonumber\\
x_{k}^{1}(n)&=&\frac{1}{n(k+1)^{1/10}},\quad
x_{k}^{2}(n)=\frac{1}{n(k+1)^{2/100}}\nonumber\\
  x_{k}^{3}(n)&=&\hspace*{-0.35cm}\frac{1}{n(k+1)^{3/100}},\quad \left\langle \beta_{1},\phi_{k}\right\rangle_{L^{2}(a,b)} =\frac{1}{(k+1)^{3/5}}\nonumber\\
\left\langle \beta_{2},\phi_{k}\right\rangle_{L^{2}(a,b)}&=&\frac{1}{(k+1)^{7/10}},\quad\left\langle \beta_{3},\phi_{k}\right\rangle_{L^{2}(a,b)}=\frac{1}{(k+1)^{4/5}}.
\label{example6}
\end{eqnarray}

 In Model 1, the fastest velocity decay of the eigenvalues of the autocovariance  and autocorrelation kernels is displayed.  The  regressor kernels define respective Hilbert-–Schmidt  integral operators. Model 2 corresponds to the most singular scenario, with the  regressors being defined by  compact but not  Hilbert--Schmidt operators.
The empirical functional mean quadratic errors, obtained from $r=100$ realizations of a functional sample of size $N=200,$ are showed, in  Table \ref{T3a} for Model 1, and in Table \ref{T6qq} for Model 2, considering the times $n=10t,$ $t=1,\dots,20,$ from the $200$ times computed.
The regularity properties, i.e., continuity and differentiability properties of the regression parameter functions, and of the autocovariance kernels  of the res-\linebreak ponse and innovations, as well as of the autocorrelation and regressor kernels, directly affect the performance of the presented  approach. For the sample sizes $N=200, 600, 1000,$ and truncation parameter values $k_{N}=2,3,4,$ tested, the best performance corresponds to   Model 1, providing the most regular parametric  scenario. The worst performance is observed in Model 2, corresponding to the most singular scenario, leading to the largest values of $\Lambda_{k_{N}}=\sup_{j=1,\dots,k_{N}}\frac{1}{\lambda_{j}(R_{0})-\lambda_{j+1}(R_{0})}.$ See
Theorem 2 of Guillas (2001), which  provides the convergence to zero of the functional mean-square error, in the norm of  $\mathcal{L}(H).$ Note that according to this result, the optimal choice of $k_{N}$ is such that $$\lambda_{k_{N}}^{4+2\gamma }(R_{0})=\frac{c\Lambda_{k_{N}}^{2}}{N^{1-2\epsilon}},\quad c>0,\ \epsilon<1/2,\ \gamma\geq 1 .$$ The rate of convergence in quadratic mean is then of order $$\lambda_{k_{N}}^{2}(R_{0})\simeq \left[\frac{\Lambda_{k_{N}}^{2}}{N^{1-2\epsilon}}\right]^{1/(\gamma +2)}$$
\noindent (see  Supplementary Material I, to compare with Models 3--6).

\small{\begin{table}[H]
\centering
\begin{tabular}{c r c r}
Time	&	EFMQE	&	Time	&	EFMQE	\\	\hline	\hline
10	&	  0.0075	&	110	&	  0.0030	\\		
20	&	 0.0072	&	120	&	   0.0038	\\		
30	&	 0.0058	&	130	&	   0.0023\\		
40	&	0.0039 &	140	&	 0.0036	\\		
50	&	 0.0048 &	150	&	 0.0018	\\		
60	&	0.0042 &	160	&	 0.0033	\\		
70	&	0.0020	&	170	&	  0.0052\\		
80	&	 0.0062	&	180	&	0.0056	\\		
90	&	 0.0036	&	190	&	  0.0023	\\		
100	&	0.0031	&	200	&	 0.0045	\\		
\end{tabular}
\caption{\emph{Model 1}. \ Empirical Functional Mean Quadratic Errors (EFMQEs), based on   $r=100$ repetitions of a functional response sample of size $N=200,$ considering the truncation order $k_{N}=4.$}\label{T3a}
\end{table}}
\small{\begin{table}[H]
\centering
\begin{tabular}{c r c r}
Time	&	EFMQE	&	Time	&	EFMQE\\	\hline	\hline
10	&	  0.2960	&	110	&	   0.0652	\\		
20	&	  0.3068	&	120	&	   0.0629	\\		
30	&	   0.2970	&	130	&	  0.0625\\		
40	&	0.3145  &	140	&	  0.0588	\\		
50	&	  0.2289 &	150	&	0.0372	\\		
60	&	 0.2491 &	160	&	  0.0655	\\		
70	&	0.2339	&	170	&	   0.0709\\		
80	&	 0.1496 &	180	&	0.1048	\\		
90	&	  0.1200	&	190	&	  0.1011	\\		
100	&	0.0922	&	200	&	  0.1237	\\		
\end{tabular}
\caption{\emph{Model 2}. Empirical Functional Mean Quadratic Errors (EFMQEs), based on   $r=100$ repetitions of a functional response sample of size $N=200,$ considering the truncation order $k_{N}=4.$}\label{T6qq}
\end{table}}
\normalsize{
\section{Application}
\label{application}
In this  section,  a panel of  small and medium size Spanish companies, in different industrial areas of the $15$ autonomous Spanish communities, in the Iberian Peninsula, is analyzed during the period $1999-2007,$ considering  $4$ industry sectors (Factories, Building, Commerce and Several).  Data were collected from the SABI (\emph{Sistema de An\'alisis de Balances Ib\'ericos}) database. The firm factor determinants of the leverage, considered in the analysis of the financing decisions,  are: \emph{Firm size}, \emph{Asset structure}, \emph{Profitability}, \emph{Growth},  \emph{Firm risk},  \emph{Age}.  Specifically, the leverage  is measured as the ratio of the total debt to the total assets; the firm size  is measured as the log of the total assets; the asset structure consists of the net
fixed assets divided by the total assets of the firm; the profitability is  computed as the ratio
between earnings before interest, taxes amortization and depreciation, and the total
assets; growth  is measured in terms of the growth of the assets, calculated
as the annual change of the total assets of the firm; the firm risk is given by the business risk,
and it is defined as the standard deviation of the earnings before the interest, and the taxes over book value
of the total assets, during the sample period; and, finally,  the age is measured as
the logarithm of the number of years that the firm has been operating.
  These firm factor determinants depend on   the Spanish community studied (spatial location  in the Iberian Peninsula), and on  the industrial area  sampled (located by the radial  argument, in the corresponding autonomous community). They are inspected during the period  1999-2007 (see Supplementary Material II, where the response, and these kernel regressors are represented for the Factory sector).  }
  {\normalsize \begin{table}[h!]
\centering \resizebox{12cm}{!}{\begin{tabular} {l||c c c c c c c c c}
SCC	&	1999	&	2000	&	2001	&	2002	&	2003	&	2004	&	2005	&	2006	&	2007	\\	\hline	\hline
1 (Galicia)	&	0.0053	&	0.0066	&	0.0136	&	0.0141	&	0.0050	&	0.0061	&	0.0095	&	0.0030	&	0.0182	\\		
2 (Asturias)	&	0.0559	&	0.0492	&	0.0285	&	0.0366	&	0.0299	&	0.0273	&	0.0198	&	0.0252	&	0.0280	\\		
3 (Cantabria)	&	0.0487	&	0.0213	&	0.0288	&	0.0384	&	0.0197	&	0.0175	&	0.0169	&	0.0146	&	0.0256	\\		
4 (P. Vasco)	&	0.0038	&	0.0051	&	0.0102	&	0.0070	&	0.0065	&	0.0035	&	0.0037	&	0.0052	&	0.0092	\\		
5 (Navarra)	&	0.0110	&	0.0127	&	0.0097	&	0.0106	&	0.0065	&	0.0088	&	0.0173	&	0.0106	&	0.0141	\\		
6 (Arag\'on)	&	0.0162	&	0.0069	&	0.0161	&	0.0208	&	0.0105	&	0.0107	&	0.0115	&	0.0078	&	0.0180	\\		
7 (Cataluña)	&	0.0058	&	0.0039	&	0.0186	&	0.0121	&	0.0043	&	0.0037	&	0.0046	&	0.0037	&	0.0204	\\		
8 (Cast. Le\'on)	&	0.0070	&	0.0052	&	0.0267	&	0.0309	&	0.0057	&	0.0061	&	0.0124	&	0.0058	&	0.0376	\\		
9 (La Rioja)	&	0.0662	&	0.0515	&	0.0237	&	0.0372	&	0.0221	&	0.0265	&	0.0585	&	0.0352	&	0.0237	\\		
10 (Extremadura)	&	0.0326	&	0.0273	&	0.0467	&	0.0501	&	0.0453	&	0.0452	&	0.0445	&	0.0417	&	0.0537	\\		
11	(Madrid)	&	0.0087	&	0.0021	&	0.0086	&	0.0057	&	0.0076	&	0.0096	&	0.0086	&	0.0059	&	0.0082	\\		
12	(Cast. Mancha)	&	0.0062	&	0.0087	&	0.0102	&	0.0220	&	0.0054	&	0.0053	&	0.0060	&	0.0036	&	0.0107	\\		
13 (C. Valenciana)	&	0.0129	&	0.0073	&	0.0104	&	0.0103	&	0.0094	&	0.0109	&	0.0179	&	0.0099	&	0.0240	\\		
14	(Andaluc\'{\i}a)	&	0.0170	&	0.0097	&	0.0249	&	0.0235	&	0.0048	&	0.0053	&	0.0085	&	0.0063	&	0.0440	\\		
15	 (Murcia)	&	0.0123	&	0.0086	&	0.0130	&	0.0137	&	0.0112	&	0.0102	&	0.0127	&	0.0057	&	0.0170	\\		
	
\end{tabular}
}
\caption{\small \emph{Factory Sector.} Mean  LOOCV errors at  each one of the Spanish  Autonomous Communities analyzed, for the  years studied in the period $1999-2007$.} \label{tab2}
\end{table}
}  Beals smoothing has been traditionally consider  in  Ecology to predict the probability of appearance of different species in the sample units (see, for example,  C\'aceres and Legendre,  2008).      The overall firm structure of the Spanish communities studied, during the temporal period analyzed, has been taken into account, in  the selection  procedure of suitable  target 'industry sub-sectors', in our implementation of Beals smoothing. Specifically, the following    target 'industry  subsectors' (i.e.,  target 'species') are  considered:  11 target industry  subsectors in Factory sector (food; beverages and tobacco; paper, cardboard, desktop and graphic arts; articles and  automotive; textile manufacture and footwear; manufacturer for construction and equipment; industry wood, cork and furniture; metal-mechanical industry; chemical and paraquímica industry; diverse industries;  information technology and  the knowledge economy), 3 target industry  subsectors  in Building sector (specialized construction activities; edification;  civil work),   9  target industry  subsectors in  Commerce sector (household items, furniture and appliances; electronic, computer and telecommunication equipment and components; hardware, glass and construction materials; machinery, furniture and equipment for agricultural and industrial activities; raw materials, agricultural, for industry and waste materials; pharmaceuticals, perfumery, clothing accessories; books and others; textile products and footwear; vehicles, motor, spare parts, fuels and lubricants), and 6 target industry  subsectors in  Several sector (hostelry; service to the company; distribution service; social service; consumer services; transport).   
    The estimated probability values (by Beals smoothing),   that  a given target industry subsector occurs in a specific sampling unit, play the role of weights, in the computation of a smoothed spatial version of the observed  
 firm leverage           
             (see mean firm leverage per community, and the Beals smoothed leverage mapping in  Supplementary Material II). 
 Spatial interpolation on a regular grid is then performed.   The proposed  functional regression model is   fitted from such  spatially interpolated and  smoothed  data set, in terms of the empirical eigenvectors and eigenvalues (see Supplementary Material II for  more details).  
Given the small functional sample size $N=9,$ and the distance between the empirical eigenvalues of the autocovariance operator of the regression residuals, associated with the ordinary  least-squared estimator (see Section \ref{esresiduals}),  only one empirical eigenvector ($k_{N}=1$) is considered in
equation (\ref{estcomponent})  (see also Bosq, 2000).  Leave One Out Cross Validation (LOOCV) is applied to check  model fitting.  The    mean  Leave One Out Cross Validation errors at the $15$ Spanish communities, for  the  years in the period $1999-2007,$   are  displayed, in  Tables
\ref{tab2}--\ref{tab5}, for the four industry sectors studied, respectively.  (The  Spanish Community Codes (SCC) are given in Table \ref{tab2}).   Note that, a worse fitting of the model is observed for $k_{N}=2$ and $k_{N}=3$ (see Supplementary Material II).
    The best results correspond to the Factory sector  followed by the Building and Commerce sectors, where the target firm subsectors seem to be  selected, according to the enterprise structure of most of the Spanish communities. While in  the Several industry sector the worst performance is observed, since this sector includes a greater diversity of industrial areas with little spatial dependence. Despite these observed Beals smoothing effects, the magnitude of  the mean LOOCV errors are quite stable through time and space (see also mean LOOCV error maps in the Supplementary Material II, for $k_{N}=1$).
  Given the absence of records in the used database, in the Building sector in Cantabria, and in the Commerce sector in  La Rioja, 
 we omit these lines, in the corresponding mean LOOCV error tables. The effect of these missing data can be observed in the mean LOOCV error maps in  Supplementary Material II.  The development of the presented approach, under  missing data, constitutes the subject of future work. 
 \section{Final comments}
\label{fc}
This paper  extends the generalized least-squared  estimation results obtained in  Ruiz-Medina  (2016), on FANOVA analysis of fixed effects models in Hilber spaces, under dependent errors. Specifically, the approach presented allows the analysis of functional responses over a period of time, under the control of kernel regressor in that period. 
While, in  Ruiz-Medina  (2016), a scalar fixed effect design  is considered, and the experiment is not running over time. In  Benhenni,  Hedli-Griche and Rachdi  (2017), a functional random design is assumed in simple regression under dependent errors. Here, a kernel random design is considered in multiple regression under dependent errors.
 Furthermore, sufficient conditions are obtained   for the explicit derivation of the generalized least-squared regression  parameter estimator, beyond the restriction, considered in Ruiz-Medina  (2016), of the spectral diagonalization of the functional parameters,  in terms of a common eigenvector system.     
In the practical implementation of the proposed methodology, a suitable orthonormal basis $\{\varphi_{k}=\psi_{k},\ k\geq 1\}$  of $H$ must be  considered. When $H$ is an element of the scale of fractional Sobolev spaces, including  $L^{2}$ spaces, wavelet bases provide unconditional bases for these spaces. In particular,  $\{\psi_{k},\  k\geq 1\}$ can be  an orthonormal wavelet basis  providing an $[s]+1$-regular multiresolution analysis of an $L^{2}$ space,  for a  suitable  $s>0,$ allowing the continuous inversion of the autocovariance operator $R_{0}.$  Here,  $[\cdot ]$ denotes the integer part.
The  simulation study highlight the interaction between the regularity properties of the functional data, and the performance of the presented approach, depending on the  truncation order and  the sample size. On the other hand, the real-data example illustrates its performance,  from very small functional sample sizes, requiring small truncation orders, after applying a suitable smoothing technique. The role of the kernel regressors  is illustrated as well.       
In our example, they  soft the  effect of industrial areas, in the representation  of the annual Beals smoothed firm leverage maps (response),  as the output of a linear filter, with input the regression parameters, incorporating  the information  from firm factor determinants (kernel regressors),   depending on the  industrial area sampled, and on the  Spanish community studied.
{\normalsize  \begin{table}[h!]
\centering \resizebox{12cm}{!}{\begin{tabular} {l||c c c c c c c c c}
SCC	&	1999	&	2000	&	2001	&	2002	&	2003	&	2004	&	2005	&	2006	&	2007	\\	\hline	\hline
1	&	0.0238	&	0.0163	&	0.0332	&	0.0359	&	0.0154	&	0.0169	&	0.0261	&	0.0378	&	0.0157	\\		
2	&	0.0628	&	0.0680	&	0.0703	&	0.0494	&	0.0715	&	0.0937	&	0.0648	&	0.0445	&	0.0557	\\	
4	&	0.0416	&	0.0301	&	0.0382	&	0.0474	&	0.0336	&	0.0165	&	0.0376	&	0.0477	&	0.0365	\\		
5	&	0.0290	&	0.0301	&	0.0261	&	0.0808	&	0.0191	&	0.0399	&	0.0898	&	0.0756	&	0.0389	\\		
6	&	0.0245	&	0.0163	&	0.0375	&	0.0370	&	0.0122	&	0.0507	&	0.0407	&	0.0480	&	0.0158	\\		
7	&	0.0148	&	0.0136	&	0.0230	&	0.0276	&	0.0195	&	0.0149	&	0.0177	&	0.0471	&	0.0216	\\		
8	&	0.0540	&	0.0538	&	0.0664	&	0.0465	&	0.0684	&	0.0314	&	0.0610	&	0.1226	&	0.0795	\\		
9	&	0.0639	&	0.0457	&	0.0636	&	0.1043	&	0.0554	&	0.0937	&	0.0599	&	0.1636	&	0.0498	\\		
10	&	0.0294	&	0.0306	&	0.0337	&	0.0311	&	0.0260	&	0.0330	&	0.0487	&	0.0689	&	0.0461	\\		
11	&	0.0199	&	0.0333	&	0.0190	&	0.0255	&	0.0143	&	0.0092	&	0.0144	&	0.0418	&	0.0147	\\		
12	&	0.0251	&	0.0248	&	0.0316	&	0.0262	&	0.0246	&	0.0315	&	0.0432	&	0.0600	&	0.0222	\\		
13	&	0.0226	&	0.0224	&	0.0300	&	0.0310	&	0.0190	&	0.0190	&	0.0190	&	0.0179	&	0.0177	\\		
14	&	0.0335	&	0.0504	&	0.0546	&	0.0620	&	0.0298	&	0.0289	&	0.0245	&	0.1275	&	0.0336	\\		
15	&	0.0316	&	0.0321	&	0.0413	&	0.0432	&	0.0092	&	0.0397	&	0.0225	&	0.0332	&	0.0560	\\		
\end{tabular}
}
\caption{\small\emph{Building Sector.}  Mean  LOOCV errors at  each one of the Spanish  Autonomous Communities analyzed, for the  years studied in the period $1999-2007$.}\label{tab3}
\end{table}
}  
 {\normalsize \begin{table}[h!]
\centering \resizebox{12cm}{!}{\begin{tabular} {l||c c c c c c c c c}
SCC	&	1999	&	2000	&	2001	&	2002	&	2003	&	2004	&	2005	&	2006	&	2007	\\	\hline	\hline
1	&	0.0094	&	0.0100	&	0.0071	&	0.0092	&	0.0089	&	0.0090	&	0.0113	&	0.0120	&	0.0078	\\		
2	&	0.0259	&	0.0258	&	0.0233	&	0.0247	&	0.0208	&	0.0250	&	0.0260	&	0.0270	&	0.0223	\\		
3	&	0.0211	&	0.0236	&	0.0153	&	0.0153	&	0.0180	&	0.0236	&	0.0274	&	0.0251	&	0.0154	\\		
4	&	0.0049	&	0.0052	&	0.0052	&	0.0047	&	0.0054	&	0.0064	&	0.0057	&	0.0064	&	0.0051	\\		
5	&	0.0879	&	0.0850	&	0.0821	&	0.0789	&	0.0833	&	0.0877	&	0.0810	&	0.0826	&	0.0794	\\		
6	&	0.0129	&	0.0172	&	0.0126	&	0.0128	&	0.0149	&	0.0166	&	0.0188	&	0.0171	&	0.0109	\\		
7	&	0.0042	&	0.0057	&	0.0045	&	0.0067	&	0.0060	&	0.0061	&	0.0048	&	0.0064	&	0.0058	\\		
8	&	0.0176	&	0.0165	&	0.0178	&	0.0175	&	0.0169	&	0.0157	&	0.0169	&	0.0148	&	0.0187	\\				
10	&	0.0084	&	0.0085	&	0.0106	&	0.0093	&	0.0082	&	0.0094	&	0.0090	&	0.0105	&	0.0097	\\		
11	&	0.0099	&	0.0101	&	0.0105	&	0.0100	&	0.0114	&	0.0130	&	0.0190	&	0.0145	&	0.0132	\\		
12	&	0.0099	&	0.0138	&	0.0068	&	0.0052	&	0.0072	&	0.0122	&	0.0183	&	0.0205	&	0.0074	\\		
13	&	0.0079	&	0.0075	&	0.0082	&	0.0079	&	0.0093	&	0.0110	&	0.0092	&	0.0082	&	0.0088	\\		
14	&	0.0236	&	0.0239	&	0.0206	&	0.0209	&	0.0235	&	0.0251	&	0.0228	&	0.0241	&	0.0197	\\		
15	&	0.0088	&	0.0090	&	0.0072	&	0.0069	&	0.0086	&	0.0110	&	0.0106	&	0.0106	&	0.0074	\\		

\end{tabular}
}
\caption{\small \emph{Commerce Sector.} Mean  LOOCV errors at  each one of the Spanish  Autonomous Communities analyzed, for the  years studied in the period $1999-2007$.}\label{tab4} 
\end{table}
} {\normalsize \begin{table}[h!]
\centering \resizebox{12cm}{!}{\begin{tabular} {l||c c c c c c c c c}
SCC	&	1999	&	2000	&	2001	&	2002	&	2003	&	2004	&	2005	&	2006	&	2007	\\	\hline	\hline
1	&	0.0578	&	0.0577	&	0.0547	&	0.0595	&	0.0296	&	0.0442	&	0.0464	&	0.0437	&	0.0526	\\		
2	&	0.0351	&	0.0085	&	0.0157	&	0.0253	&	0.1956	&	0.0228	&	0.0157	&	0.0341	&	0.0440	\\		
3	&	0.0360	&	0.0385	&	0.0354	&	0.0334	&	0.3637	&	0.0357	&	0.0480	&	0.0406	&	0.0449	\\		
4	&	0.0190	&	0.0257	&	0.0214	&	0.0341	&	0.2277	&	0.0197	&	0.0191	&	0.0253	&	0.0307	\\		
5	&	0.0674	&	0.0379	&	0.0397	&	0.0711	&	0.2124	&	0.0416	&	0.0407	&	0.0389	&	0.0472	\\		
6	&	0.0207	&	0.0311	&	0.0376	&	0.0578	&	0.7336	&	0.0279	&	0.0363	&	0.0298	&	0.0305	\\		
7	&	0.0440	&	0.0401	&	0.0109	&	0.0373	&	0.0876	&	0.0192	&	0.0232	&	0.0351	&	0.0371	\\		
8	&	0.0215	&	0.0264	&	0.0137	&	0.0714	&	0.5700	&	0.0308	&	0.0136	&	0.0204	&	0.0202	\\		
9	&	0.0406	&	0.0592	&	0.0689	&	0.0707	&	0.2736	&	0.0732	&	0.0560	&	0.0533	&	0.0631	\\		
10	&	0.0464	&	0.0479	&	0.0315	&	0.1038	&	0.1239	&	0.0416	&	0.0364	&	0.0450	&	0.0514	\\		
11	&	0.0647	&	0.0259	&	0.0333	&	0.0292	&	0.0718	&	0.0259	&	0.0183	&	0.0418	&	0.0433	\\		
12	&	0.0273	&	0.0288	&	0.0206	&	0.0465	&	0.1548	&	0.0556	&	0.0243	&	0.0569	&	0.0532	\\		
13	&	0.0190	&	0.0330	&	0.0315	&	0.0554	&	0.4012	&	0.0475	&	0.0399	&	0.0398	&	0.0392	\\		
14	&	0.0624	&	0.0092	&	0.0223	&	0.0237	&	0.2590	&	0.0245	&	0.0351	&	0.0307	&	0.0483	\\		
15	&	0.0247	&	0.0346	&	0.0116	&	0.0240	&	0.3455	&	0.0468	&	0.0277	&	0.0848	&	0.0948	\\		
	\end{tabular}
}
\caption{\small \emph{Several Sector.} Mean  LOOCV errors at  each one of the Spanish  Autonomous Communities analyzed, for the  years studied in the period $1999-2007$.} \label{tab5}
\end{table}
}

\begin{acknowledgements}
This work has been supported in part by project
MTM2015-71839-P of MINECO, Sapin (co-funded with FEDER funds). D. Miranda supported by  FINCyT, Inn\'ovate Per\'u
\end{acknowledgements} 

\begin{thebibliography}{9}

\bibitem{Aneiros06} 
 Aneiros-P\'erez G, Vieu P (2006) Semi-functional partial linear
regression. Stat. Probab. Letters 76:1102--1110

\bibitem{Aneiros08}
 Aneiros-P\'erez G, Vieu P (2008) Nonparametric time series
prediction: A semi-functional partial linear modeling. J.
Multivariate Anal. 99:834--857

\bibitem{Benhenni}
 Benhenni K,  Hedli-Griche S,  Rachdi M (2017) Regression models with correlated errors based on functional random design. Test 26:1--21
\bibitem{Bosq00}
Bosq D (2000) Linear Processes in Function Spaces.
Springer-Verlag, New York


\bibitem{BosqRuizMedina14}
Bosq D, Ruiz-Medina MD (2014) Bayesian estimation in a high dimensional parameter framework. Electronic Journal of Statistics 8:1604--1640

\bibitem{Caceres08}
 C\'aceres MD, Legendre P (2008)   Beals smoothing revisited. Oecologia 156:657--669


\bibitem{Cai06}
    Cai T, Hall P (2006)  Prediction in functional linear regression. Annals of Statistics 34:2159--2179

\bibitem{Chaouch17}
 Chaouch M,  Laib N,  Louani, D (2017) Rate of uniform consistency for a class of mode regression on functional stationary ergodic data.
Statistical Methods \& Applications 26:19--47


\bibitem{Chiou04}
   Chiou  J,  M\'uller HG,  Wang JL (2004) Functional response models.  Statistica Sinica 14:  659--677

\bibitem{Crambes09}
    Crambes C,  Kneip A, Sarda P (2009)  Smoothing splines estimators for functional linear regression. Annals of Statistics 37:35--72

\bibitem{Cuevas02}
Cuevas A,   Febrero M,  Fraiman R (2002) Linear functional regression: The case
of a fixed design and functional response. Canadian J. Statistics 
30:285--300

\bibitem{Cuevas14}
 Cuevas A (2014) A partial overview of the theory of statistics
with functional data. Journal of Statistical Planning and Inference 147:1--23
\bibitem{DautrayLions85}
Dautray R, Lions JL (1985) Mathematical Analysis and Numerical Methods
for Science and Technology, Vol. 3, Spectral Theory and Applications. Springer, New York

\bibitem{PratoZabczyk}
Da Prato G.,  Zabczyk J. (2002)
Second Order Partial Differential
Equations in Hilbert Spaces. University Press, Cambridge


\bibitem{EspejoFerRu17}
Espejo RM, Fern\'andez-Pascual R, Ruiz-Medina MD (2017)
Spatial-depth functional estimation of ocean temperature
from non-separable covariance models. Stoch. Environ. Res. Risk 
Assess. 31:39--51


\bibitem{FebreroBande15}
 Febrero-Bande M,  Galeano P, Gonzalez-Manteiga W (2015)
Functional principal component regression and functional 
partial least-squares regression: an overview and a comparative study. International Statistical
Review doi:10.1111/insr.12116

\bibitem{FerratyGoia13} 
Ferraty F,   Goia  A,   Salinelli E,  Vieu P (2013) Functional
projection pursuit regression. Test 22:293--320

\bibitem{Ferraty02}
 Ferraty F, Goia  A,   Vieu, P (2002) Functional nonparametric model for time series: a fractal approach for dimension reduction. Test 
 11:317--344


\bibitem{FerratyKeilegom12}
 Ferraty  F,  Keilegom IV, Vieu P (2012) Regression when both
response and predictor are functions. J. Multivariate Anal.
109:10--28

\bibitem{Ferraty06}
 Ferraty F, Vieu P (2006) Nonparametric Functional Data Analysis: Theory and Practice.  Springer, New York

\bibitem{Ferraty11} 
 Ferraty F, Vieu P (2011) Kernel regression estimation for
functional data. In: Ferraty F, Romain Y (eds) The Oxford
Handbook of Functional Data Analysis.  Oxford University Press,
Oxford, pp. 72--129

\bibitem{Fitzmaurice04}
Fitzmaurice GM, Laird NM, Ware JH (2004) Applied Longitudinal Analysis.  John
Wiley and Sons, New York

\bibitem{Geenens11}
 Geenens G (2011) Curse of dimensionality and related
issues in nonparametric functional
regression.  Statistics Surveys 5:30--43

\bibitem{GoiaVieu15}
 Goia A,  Vieu P (2015) A partitioned single functional index model. Computational Statistics 30:673--692.

\bibitem{GoiaVieu16}
 Goia A,  Vieu P (2016) An introduction to recent advances in high/infinite dimensional statistics. Journal of Multivariate Analysis 146:1--6.

\bibitem{Guillas01}
 Guillas S (2001) Rates of convergence of autocorrelation estimates for
autoregressive Hilbertian processes. Statistics \& Probability
Letters 55:281--291

\bibitem{HorvathandKokoszka}
Horv\'ath L,  Kokoszka P (2012) Inference for Functional Data with Applications. Springer, New York

\bibitem{HsingEubank15}
 Hsing T, Eubank  R (2015) Theoretical Foundations of Functional Data Analysis, with an Introduction to Linear Operators. In: Wiley Series in Probability and Statistics, John Wiley \& Sons, Chichester

\bibitem{Kara-Zaitri17}
 Kara L Z,  Laksaci A,  Rachdi M,  Vieu P (2017a) 
Uniform in bandwidth consistency for various kernel estimators involving functional data. Journal of Nonparametric 
Statistics  29:85--107

\bibitem{Kara-Zaitri17b}
 Kara L Z,  Laksaci A,  Rachdi M,   Vieu P (2017b)
Data-driven kNN estimation in nonparametric functional
data analysis. Journal of Multivariate Analysis 153:176--188

\bibitem{Ling17}
 Ling N,  Liu Y,   Vieu P (2017) On asymptotic properties of functional conditional mode estimation with both stationary ergodic and responses MAR. In Functional Statistics and Related Fields, pp 173-178, Springer,  
Switzerland

\bibitem{Marx99}
 Marx BD, Eilers PH (1999)  Generalized linear regression on sampled signals and
curves: A P -spline approach. Technometrics 41:1--13

\bibitem{Mas04}
 Mas A  (2004) Consistance du pr\'edicteur dans le mod\`ele ARH(1): le cas
compact. Ann. I.S.U.P. 48:39--48

\bibitem{Mas07}
Mas A (2007)  Weak-convergence in the functional autoregressive model.
J. Multivariate Anal. 98:1231--1261



\bibitem{Morris15}
Morris JS (2015) Functional regression. Annual Review of Statistics and Its
Application 2:321--359


\bibitem{Ramsay05}
 Ramsay JO and  Silverman BW (2005) Functional data analysis,
Second Ed. Springer Series in Statistics. Springer, New York


\bibitem{RuizMedina11}
Ruiz-Medina MD (2011) Spatial autoregressive and   moving
average Hilbertian processes. Journal of Multivariate
Analysis 102:292--305

\bibitem{RuizMedina12a}
Ruiz-Medina MD (2012a)  New challenges in spatial and spatiotemporal functional statistics for high-dimensional data.  Spatial Statistics 
1:82--91

\bibitem{RuizMedina12b}
 Ruiz-Medina MD (2012b) Spatial functional prediction from spatial
autoregressive Hilbertian processes. Environmetrics 23:119--128

\bibitem{RuizMedina16}
 Ruiz-Medina MD (2016)  Functional analysis of variance for Hilbert-valued multivariate fixed effect models. Statistics 50:689--715






\end{thebibliography}

\end{document}